\documentclass[a4paper,11pt,reqno]{amsart}
\usepackage{a4wide,amssymb,latexsym}
\usepackage{eucal}

\usepackage{amsmath}
\usepackage{amsthm}
\usepackage{amstext}
\usepackage{amsfonts}
\usepackage{dsfont}
\usepackage{amssymb}
\usepackage{graphicx}
\usepackage{epsfig}

\usepackage[usenames,dvipsnames]{pstricks}
\usepackage{epsfig}
\usepackage{pst-grad} 
\usepackage{pst-plot} 
\numberwithin{equation}{section}

\newcommand{\R}{\mathbb R}

\newcommand{\N}{\mathbb N}

\newcommand{\be}{\begin{equation}}
\newcommand{\ee}{\end{equation}}
\newcommand{\ba}{\begin{eqnarray}}
\newcommand{\ea}{\end{eqnarray}}

\newcommand{\cI}{\mathcal I}

\newcommand{\cN}{\mathcal N}

\newcommand{\cT}{\mathcal T}

\def\R{\mathbb{R}}
\def\N{\mathbb{N}}
\def\Cc{\mathcal{C}}

\def\tu{\tilde{u}}
\def\tv{\tilde{v}}
\def\ot{(0,T)\times (0,1)}
\def\Ot{[0,T]\times [0,1]}
\def\Lip{\mathrm{Lip}}
\def\ra{\rightarrow}
\newcommand{\COI}{\ensuremath{\mathcal{C}^0([0,T];\mathcal{C}^1([0,1]))}}

\newtheorem{prop}{Proposition}[section]
\newtheorem{thm}{Theorem}
\newtheorem{defi}{Definition}
\newtheorem{lem}{Lemma}
\newtheorem{Rk}{Remark}
\newtheorem{Co}{Corollary}

\begin{document}

\title{Finite-time stabilization of systems of conservation laws on networks}

\author{Vincent Perrollaz}
\address{Laboratoire de Math\'ematiques et Physique Th\'eorique, 
Universit\'e de Tours,
UFR Sciences et Techniques,
Parc de Grandmont,
37200 Tours, France}
\email{Vincent.Perrollaz@lmpt.univ-tours.fr}

\author{Lionel Rosier}
\address{Institut Elie Cartan, UMR 7502 UdL/CNRS/INRIA,
B.P. 70239, 54506 Vand\oe uvre-l\`es-Nancy Cedex, France}
\email{Lionel.Rosier@univ-lorraine.fr}



\begin{abstract} 
We investigate the finite-time boundary stabilization of a 1-D first order quasilinear hyperbolic system of diagonal form on [0,1]. The dynamics of both
boundary controls are governed by a finite-time stable ODE. The solutions of the closed-loop system issuing from small initial data in Lip([0,1]) are shown to exist for all times and 
to reach the null equilibrium state in finite time. When only one boundary feedback law is available, a finite-time stabilization is shown to occur 
roughly in a twice longer time. The above feedback strategy is then applied to the Saint-Venant system for the regulation of water flows in a network of canals. 
\end{abstract}

\maketitle
\section{Introduction}
  
Solutions of certain asymptotically stable ODE may reach the equilibrium state in finite time. This phenomenon, which is common when using 
feedback laws that are not Lipschitz continuous, was termed {\em finite-time stability} in \cite{BB} and investigated in that paper. 

A finite-time stabilizer is a feedback control for which the closed-loop system is finite-time stable around some equilibrium. 
In some sense, it satisfies a controllability  objective with a control in feedback form. On the other hand, a finite-time stabilizer may 
be seen as an exponential stabilizer
yielding an arbitrarily large decay rate for the solutions to the closed-loop system. This explains why 
some efforts were made in the last decade to construct finite-time stabilizers for controllable systems, including  the linear ones. 
See  \cite{MP1,MP2} for some recent developments
and up-to-date references, and \cite{BR} for some connections with Lyapunov theory. 

For PDEs, the relationship between exact controllability and rapid stabilization was investigated in \cite{slemrod,komornik,komornik-book}. (See also
\cite{LRZ} for the rapid semiglobal stabilization of the Korteweg-de Vries equation using a time-varying feedback law.)
 
To the best knowledge of the authors, the analysis of the finite-time stabilization of PDE is not developed yet. However, the phenomenon of finite-time extinction 
exists naturally for certain nonlinear evolution equations (see \cite{QW,diaz,CG}). On the other hand, it is well-known since \cite{majda} that solutions of
the wave equation on a bounded domain may disappear when using ``transparent'' boundary conditions. For instance, the solution of the 1-D wave equation 
\ba
\partial _t^2 y-\partial_x^2y=0,&&  \text{in }  (0,T)\times (0,1), \label{Int1}\\
\partial _x y(t,1)= - \partial _t y(t,1),&& \text{in }  (0,T),  \label{Int2}\\
\partial _x y(t,0)=\partial _t y(t,1),&& \text{in }  (0,T),  \label{Int3}\\
(y(0,.),\partial _t y(0,.))=(y_0,z_0), &&\text{in } (0,1),\label{Int4}
\ea
is finite-time stable in $\{ (y,z)\in H^1(0,1)\times L^2(0,1); \ y(0)+y(1)+\int_0^1z(x)dx=0\}$, with $T=1$ as extinction time (see e.g. \cite[Theorem 0.5]{komornik-book}
for the details.) The condition \eqref{Int2} is transparent in the sense that a wave $y(t,x)=f(x-t)$ traveling to the right satisfies \eqref{Int2} and leaves the domain
at $x=1$ without generating any reflected wave.  Note that we can replace \eqref{Int3} by the boundary condition $y(t,0)=0$ (or $\partial _x y(t,0)=0$). Then a finite-time
extinction still occurs (despite the fact that waves bounce at $x=0$) with an extinction time $T=2$.  We refer to \cite{CZ} for the analysis of the finite-time 
extinction property
for a nonhomogeneous string with a viscous damping at one extremity, and to \cite{APR} for the investigation of the finite-time stabilization of a network of strings. 

The finite-time stability of \eqref{Int1}-\eqref{Int4} is easily established when writing \eqref{Int1} as a first order hyperbolic system
\[
\partial _t \left( \begin{array}{l}r\\s\end{array} \right)  -\partial _x \left( \begin{array}{l} s\\r \end{array}\right) =0
\]
with $(r,s)=(\partial _x y,\partial _t y)$, and next introducing the Riemann invariants $u=r-s$, $v=r+s$ that solve the system of two transport equations
\begin{eqnarray*}
&&\partial _t u+ \partial _x u =0,\label{Int11} \\
&&\partial _t v-  \partial _x v =0.\label{Int12}
\end{eqnarray*} 
The boundary conditions \eqref{Int2} and  \eqref{Int3} yield $u(t,0)=v(t,1)=0$ (and hence $u(t,.)=v(t,.)=0$ for $t\ge 1$), while the boundary conditions
\eqref{Int2} and $y(t,0)=0$ yield $v(t,1)=0$ and $u(t,0)=v(t,0)$ (and hence $v(t,.)=0$ for $t\ge 1$ and $u(t,.)=0$ for $t\ge 2$).

The goal of this paper is to show that the finite-time extinction property can be realized for 1-D first order quasilinear hyperbolic systems
\be
\partial_t Y +\partial _x F(Y) =0,\label{Int12bis}
 \ee
that can be put in diagonal form, i.e. for which there is a smooth change of (dependent) variables that transforms \eqref{Int12bis} into a system of two 
nonlinear transport equations of the form
\ba
&&\partial _t u +\lambda (u,v)\partial _x u=0,\label{Int13}\\
&&\partial _t v + \mu (u,v) \partial _x v=0,\label{Int14}
\ea
where $\mu (u,v)\le -c < c \le\lambda (u,v)$ are smooth functions and $c>0$ is some constant. 
In practice, the functions $u$ and $v$ are  Riemann invariants of \eqref{Int12bis} 
(see e.g. \cite{evans}). 

The generalization of the finite-time extinction property of the wave equation to systems of the form \eqref{Int13}-\eqref{Int14} is the main aim of this paper.
Of course, one could just consider homogeneous Dirichlet conditions
\[
u(t,0)=v(t,1)=0,
\]
but this would impose to restrict ourselves to initial data $(u_0,v_0)$ fulfilling the compatibility conditions 
\[
u_0(0)=v_0(1)=0.
\]
Rather, we shall consider boundary conditions whose dynamics obey a finite-time stable ODE, namely
\begin{eqnarray}
&&\frac{d}{dt} u(t,0) = -K\text{sgn}(u(t,0))|u(t,0)|^\gamma,  \label{Int21}\\ 
&&\frac{d}{dt} v(t,1) = -K\text{sgn}(v(t,1))  |v(t,1)|^\gamma ,\label{Int22}
\end{eqnarray} 
($(K,\gamma )\in (0,+\infty ) \times (0,1)$ being some constants)
and supplement the system \eqref{Int13}-\eqref{Int14}, \eqref{Int21}-\eqref{Int22} with the initial condition
\be
u(0,x)=u_0(x),\ v(0,x)=v_0(x). \label{Int23}
\ee
The first main result in this paper (Theorem \ref{main1}) asserts  that for any pair $(u_0,v_0)$ of (small enough) Lipschitz continuous initial data, 
system \eqref{Int13}-\eqref{Int14} and \eqref{Int21}-\eqref{Int23} admits
a unique solution in some class of Lipschitz continuous functions,  and that this solution is defined for all times $t\ge 0$ and vanishes
for roughly $t\ge 1/c$.  Theorem \ref{main1} is proved by using a fixed-point argument (Schauder Theorem) and energy estimates.

Sometimes, the boundary condition at one extremity of the domain (say 0) is imposed by the context, so that we cannot 
chose the condition $u(t,0)=0$ (or its generalization \eqref{Int21}) for the Riemann invariant $u$. Then, we have to replace
\eqref{Int21} by a boundary condition of the form
\be
u(t,0)=h(v(t,0),t),
\label{Int21bis}
\ee 
for some (smooth) function $h=h(v,t)$. The second main result in this paper (Theorem \ref{main2}) asserts that the system 
\eqref{Int13}-\eqref{Int14} and \eqref{Int22}-\eqref{Int21bis} is still locally well-posed with roughly an extinction time
$T=2/c$. The result is obtained for small initial data and for $||\partial_t h||_\infty$ small enough. 

The results obtained in this paper can be applied to: 
\begin{enumerate}
\item the $p-$system
\ba
&&\partial _t r  -\partial _x s=0,\\
&&\partial _t s-\partial _x [p(r)]=0,
\ea
where $p\in C^1(\R )$ is any given function;
\item the shallow water  equations (also called  Saint-Venant  equations \cite{StVenant})
\ba
&&\partial _t H  +\partial _x (HV)=0,\\
&&\partial _t V+\partial _x ( \frac{V^2}{2} +gH)=0,
\ea
where $H$ is the water depth and $V(t,x)$ the averaged horizontal velocity of water in a canal, and $g$ the gravitation constant;
\item Euler's equations for barotropic compressible gas
 \ba
&&\partial _t \rho   + \partial _x (\rho V)=0,\\
&&\partial _t (\rho V)+\partial _x ( \rho V^2  +p)=0,
\ea
where $\rho$ is the mass density, $V$  the velocity, and $p=p(\rho )$  the pressure of the gas.
\item The same strategy could in theory be applied to any system possessing Riemann invariants.
Riemann invariants exist for most $2\times 2$ systems, and also for some larger systems 
(e.g. the $3\times 3$ system of Euler's equations for compressible gas, see \cite[chapters 18,20]{Smoller}). 
\end{enumerate}
For the sake of shortness, we will limit ourselves to the stabilization of Saint-Venant equations, and will give an extension of the above finite-time stabilization results to 
a tree-shaped network of canals. The obtained extinction time will be roughly $d/c$, where $d$ denotes the depth of the tree
(Theorem \ref{main5}). 

There is a huge literature about the controllability and stabilization of first order hyperbolic equations (see e.g. \cite{GL,glass,li,
perrollaz1,LRW,GH,GDL}). In particular, the control of Saint-Venant equations has attracted the attention of the control
 community because of its relevance to the regulation of water flows in networks of canals or rivers. 
We refer the reader to  e.g. \cite{CAB, XS,LS,HPCAB,GL1,BCA,DBCA,GL2,BC}, where Riemann invariants played often a great role in the design of the controls. 
Our main contribution here is to notice that a finite-time stabilization can be achieved as well, i.e. that bounces of waves at the two ends of the domain  can
be avoided.  

A numerical scheme and some numerical experiments  for the finite-time stabilization of water flows in a canal  may be found in \cite{PRIFAC}, in which certain
results of this paper were announced.

The paper is outlined as follows. Classical but important properties of linear transport equations are recalled in Section 2. 
In Section 3, we introduce two boundary controls whose dynamics are governed by a finite-time stable ODE,
and prove the existence and uniqueness of a  solution  to the closed-loop system, and the fact that this solution 
reaches the null state in finite time. In Section 4, we investigate the same problem with only one boundary
control, the other boundary condition being imposed by the physical context. In the last section, we apply the results in Sections 3 and 4 to the regulation of water flows 
in a canal with one or two boundary controls, and extend the finite-time stabilization results to any tree-shaped network of canals.

\section{Some background about linear transport equations}

\subsection{Notations}

$\Cc ^0([0,T]\times [0,1])$ denotes the space of continuous functions $u:[0,T]\times [0,1]\to \R$. It is endowed with the norm
\[
||u||_{\Cc ^0([0,T]\times [0,1])} = \sup_{(t,x)\in \Ot }|u(t,x)|.
\]
The norm of the space $L^p(0,1)$ is denoted $||\cdot ||_p$ for $1\le p\le \infty$. $\text{Lip}([0,1])$ denotes the space of Lipschitz continuous
functions $u:[0,1]\to \R$. It may be identified with the Sobolev space $W^{1,\infty}(0,1)$. $\text{Lip}([0,1])$ is endowed with the
$W^{1,\infty}(0,1)$-norm; that is
\[
||u||_{\text{Lip}([0,1])} =||u||_{W^{1,\infty}(0,1)} = ||u||_\infty + ||u'||_{\infty} \cdot
\]  
We use similar norms for $\text{Lip}(\R)$, $\text{Lip}([0,T]\times [0,1])$, etc.
\subsection{Linear transport equation}

In this section we consider the initial boundary-value problem for the following linear transport equation
\begin{equation}
\partial_t y+a(t,x)\partial_x y=0.  \label{eq:TR}
\end{equation}
We assume thereafter that 
\begin{eqnarray}
&& a\in \Cc ^0 ([0,T]\times [0,1])\cap L^\infty(0,T; \Lip ([0,1])), \label{A3}\\
&& a(t,x) \ge c>0, \qquad \forall (t,x)\in [0,T]\times [0,1], \label{A4}
\end{eqnarray}
where $c$ denotes some constant.
Note that the case when $a(t,x)\leq -c<0$ can be reduced to \eqref{A4} by the transformation $x\ra 1-x$.

\subsection{Properties of the flow}
\label{subsec:flow}

By \eqref{A3}, $a$ is uniformly Lipschitz continuous in the variable $x$, with say a Lipschitz constant $L=||a||_{L^{\infty}(0,T ; \Lip([0,1]))}$. 
Since we intend to use the method of characteristics to solve \eqref{eq:TR}, we need to study the flow associated with  $a$.
\begin{defi}
For $(t,x)\in [0,T]\times [0,1]$, let $\phi(.,t,x)$ denote the $\mathcal{C}^1$ maximal solution to the Cauchy problem
\begin{equation}
\begin{cases} 
\partial_s \phi(s,t,x)=a(s,\phi(s,t,x)), \\
\phi(t,t,x)=x,
\end{cases}
\end{equation}
 which is defined on a certain subinterval $[e(t,x),f(t,x)]$ of $[0,T]$ (which is closed since $[0,1]$ is compact), and with possibly $e(t,x)$ and/or $f(t,x)=t$.
 Let 
 \[
 \text{Dom}\,  \phi =\{  (s,t,x);\ (t,x) \in [0,T] \times [0,1],\  s\in [ e(t,x),f(t,x) ]  \}
 \]
 denote the {\em domain} of $\phi$. 
\end{defi}
 Note that 
 \begin{equation}
 e(t,x)>0 \Rightarrow \phi(e(t,x),t,x)=0.
 \label{A8}
 \end{equation}
We take into account the influence of the boundaries by introducing the sets
\begin{eqnarray*}
P&:=&\{(s,\phi(s,0,0));\  \ s \in[0,f(0,0)]\},\\
I&:=& \{(t,x)\in  [0,T]\times  [0,1] \setminus P; \ e(t,x)=0 \}, \\
J&:=& \{(t,x)\in [0,T]\times  [0,1] \setminus P; \ \phi(e(t,x),t,x)=0 \}.
\end{eqnarray*}
(See Figure \ref{fig1}.) Note that both $I$ and $J$ are open in $[0,T]\times [0,1]$. 
\begin{figure}[http]
\begin{center}
\includegraphics[scale=0.5]{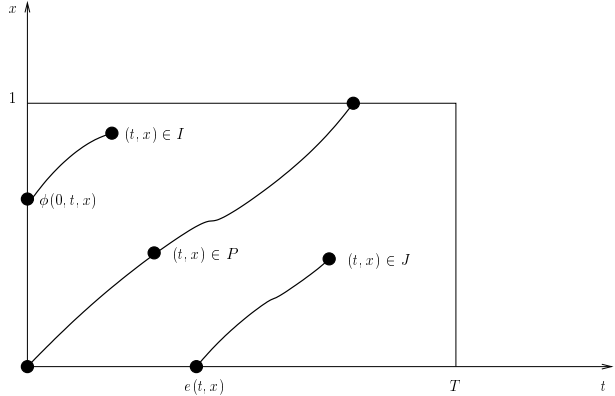}
\end{center}
\caption{Partition of $[0,T]\times [0,1]$ into $I\cup P\cup J$.}
\label{fig1}
\end{figure}

\begin{prop}\label{prop:Lip}
Let $a$ satisfying \eqref{A3},  let $L=||a||_{L^{\infty}(0,T ; \Lip([0,1]))}$, and let
$$K:=\max(1,||a||_{\Cc^0(\Ot)})e^{LT}.$$
Then $\phi$ is $K$-Lipschitz on its domain; that is, for all $(s_1,t_1,x_1),(s_2,t_2,x_2)\in \text{Dom}\, \phi$
\be 
\label{WW0}
|\phi(s_1,t_1,x_1) -\phi (s_2,t_2,x_2)|\le K \left( |s_1-s_2| + |t_1-t_2| + |x_1-x_2| \right) .
\ee
\end{prop}
The proof of Proposition \ref{prop:Lip} is given in appendix, for the sake of completeness. 

We can now study the regularity of $e$.
\begin{prop}
\label{prop:regu} Let $a$ be as in Proposition \ref{prop:Lip}, let  $(t,x)\in [0,T] \times [0,1] $, let $\{ a_n \} \subset    \mathcal{C}^0(\Ot)\cap
 L^\infty(0,T;\Lip([0,1]))$ be    a   sequence    such   that  $||a_n||_{L^\infty(0,T;\Lip([0,1]))}$  is  bounded and 
 $$||a_n-a||_{\mathcal{C}^0(\Ot)}\rightarrow      0\quad \text{ as } n\ra +\infty,$$
  and  let $\{ (t_n,x_n)\} \subset  \Ot$ be a sequence such that $(t_n,x_n)\rightarrow (t,x)$. Then
 \be
\label{A11}
 e_n(t_n,x_n)\rightarrow e(t,x).
 \ee
\end{prop}
\begin{proof}
We use again the extension operator $\Pi$ introduced in the proof of Proposition \ref{prop:Lip} (see the appendix) and set $\tilde{a}_n=\Pi(a_n)$ and $\tilde{a}=\Pi(a)$.
Let $\tilde{\phi}_n$ and $\tilde{\phi}$ denote their respective flows. Recall that $\tilde{\phi}$ and $\phi$ coincide on $\text{Dom}\, \phi$ (resp.
 $\tilde{\phi}_n$ and $\phi _n$ coincide on $\text{Dom} \, \phi _n$). From 
 \begin{eqnarray*}
 |\partial _s [\tilde \phi _n (s,t,x) -\tilde\phi (s,t,x)] | &=& |\tilde a_n (s,\tilde\phi _n (s,t,x)) - \tilde a (s,\tilde\phi (s,t,x))| \\
 &\le& |\tilde a_n (s,\tilde\phi _n (s,t,x)) - \tilde a (s,\tilde \phi _n (s,t,x))| \\
 &&\qquad    + |\tilde a (s, \tilde \phi _n (s,t,x)) -\tilde a( s , \tilde \phi (s,t,x)) |\\
 &\le& ||\tilde a _n -\tilde a ||_{L^\infty  (\R ^2 )}   + || \tilde a||_{L^\infty(\R ;\text{Lip} (\R )) }  | \tilde \phi _n(s,t,x)  -\tilde \phi  (s,t,x)|,    
 \end{eqnarray*}
\eqref{ext1}, \eqref{ext2}  and Gronwall's lemma, we infer that for all $n\ge 0$ and all $(s,t,x)\in [0,T]^2\times [0,1]$, we have
\begin{equation}
|(\tilde{\phi}_n-\tilde{\phi})(s,t,x)|\leq T  || a_n-  a ||_{\Cc ^0 (\Ot ) } e^{T||a||_{L^\infty(0,T;\Lip([0,1]))}}.
\label{eq:Gron}
\end{equation}
It may be seen that
$$e_n(t_n,x_n)=\min\{s\in[0,t_n];  \ \forall r\in [s,t_n], \ \  \tilde{\phi}_n(r,t_n,x_n)\in[0,1]\}.$$
\begin{itemize}
\item If $(t,x)\in I$, then since we have excluded the characteristic coming from $(0,0)$, we have that
$$\underset{s\in[0,t]}{\inf} \mathrm{dist}( (s,\phi(s,t,x)),[0,t]\times \{0\})>0,$$ 
where $\mathrm{dist} ((t,x),F)=\inf_{(t',x')\in F} ( |t-t'|+|x-x'| )$.
So we infer from \eqref{eq:Gron} that for $n$ large enough $\phi_n(.,t,x)$ is defined on $[0,t]$, i.e. $e_n(t,x)=0$. Then \eqref{A11} is obvious.
\item  From now on, we assume that  $(t,x)\in J \cup P$. 
We claim that 
\be
\label{A12}
\limsup_{n\to \infty}  e_n(t_n,x_n) \leq e(t,x).
\ee
Indeed, if $e(t,x)=t$, then
$$\limsup_{n\to \infty}  e_n(t_n,x_n)\leq \limsup_{n\to \infty}  t_n=t=e(t,x).$$
Otherwise, we have $e(x,t)<t$, and using \eqref{A4} we obtain for any $\epsilon \in (0, (t-e(t,x))/2)$, 
\be
c \epsilon \leq \phi(s,t,x)\leq 1-c \epsilon ,\quad
\ \forall s\in [e(t,x)+\epsilon,t-\epsilon].
\label{A13}\ee
However, we have for $n$ large enough
\begin{gather}
  ||\tilde \phi_n-\tilde \phi||_{\Cc^0 ([0,T]^2\times [0,1] ) }\leq \frac{c \epsilon}{4}, \label{A14}\\
|\tilde \phi_n(s,t_n,x_n)-\tilde \phi_n(s,t,x)|\leq \frac{c \epsilon}{4}, \label{A15}
\end{gather}
the second estimate coming from the uniform bound on $||a_n||_{L^\infty(0,T;\Lip([0,1]))}$ and Proposition \ref{prop:Lip}.
Combining \eqref{A13}, \eqref{A14} and \eqref{A15},  we see that for $n$ large and for all $s\in [e(t,x)+\epsilon, t-\epsilon ]$, 
$\phi _n (s ,t_n, x_n)$ is well defined and 
$$\phi_n(s,t_n,x_n)\geq \frac{c \epsilon}{2}.$$ 
This yields $\limsup_{n\to \infty}  e_n(t_n,x_n)\leq e(t,x)+\epsilon$, and since $\epsilon$ was arbitrarily small, \eqref{A12} follows.

If $(t,x)\in P$ the proof of \eqref{A11} is complete, for $\liminf_{n\to \infty}  e_n(t_n,x_n)\geq 0=e(t,x)$. Assume finally that $(t,x)\in J$, so that $e(t,x)>0$.
Pick any $s \in (0,e(t,x))$. We obviously have $\tilde{\phi}(s,t,x)<0$, thanks to the lower bound 
on $\tilde{a}$ (see \eqref{ext3}). But we know from  \eqref{eq:Gron} and Proposition \ref{prop:Lip}  that 
$$\tilde{\phi}_n(s,t_n,x_n)\underset{n\ra +\infty}{\ra} \tilde{\phi}(s,t,x),$$ and hence for $n$ large enough, $\tilde{\phi}_n(s,t_n,x_n)<0$ 
and $s<e(t_n,x_n)$.  Thus, we conclude that $\liminf_{n\to \infty}  e_n(t_n,x_n)\geq s$. As $s$ was arbitrarily close to $e(t,x)$, we end up with 
$$\liminf_{n\to \infty}  e_n(t_n,x_n) \geq e(t,x).$$
\end{itemize}
The proof of \eqref{A11} is complete.
\end{proof}
\begin{Rk}
 \begin{enumerate}
\item For $a_n=a$,  this shows that $e$ is continuous on $\Ot$.
  \item Since $[0,T]\times[0,1]$ is compact, Proposition \ref{prop:regu} implies that $e_n$ converges uniformly toward $e$ on $\Ot$.
 \end{enumerate}
\end{Rk}
\begin{prop}
\label{prop:reg} If, in addition to \eqref{A3}-\eqref{A4}, we have  that
$\partial_x a \in \mathcal{C}^0(\Ot)$, then $\phi$ is $\mathcal{C}^1$ on $\text{Dom } \phi$ and $e$ is $\mathcal{C}^1$ on $\Ot \setminus P$, with
for $(t,x)\in J$
\begin{equation}
\label{XYZ}
 \partial_t e(t,x)=\frac{a(t,x)\exp(-\int_{e(t,x)}^t{\!\!\partial_x a(r,\phi(r,t,x))dr})}{a(e(t,x),0)} , \
 \partial_x e(t,x)=-\frac{\exp(-\int_{e(t,x)}^t{\!\!\partial_x a(r,\phi(r,t,x))dr})}{a( e(t,x),0 )} \cdot
\end{equation}
\end{prop}
\begin{proof}
The regularity of $\phi$ is a classical result (see e.g. \cite{hartman}). If $(t,x)\in I, \ e(t,x)=0$ 
and the result is obvious. For $(t,x)\in J\cap \ot$ we have $\phi(e(t,x),t,x)=0$ and $\partial_s\phi(e(t,x),t,x)>0$, therefore the Implicit Function Theorem allows us conclude.
Finally, for $(t,x)\in J\setminus \ot$, it is sufficient to pass to the limit in \eqref{XYZ}.
\end{proof}
\begin{prop}
 \label{prop:regBis}
Let $a$ fulfill \eqref{A3} and \eqref{A4}, and let $L= ||a||_{L^\infty (0,T;\text{Lip} ( [0,1] ) ) } $.
Then the function $e$ is $\bar{K}$-Lipschitz on $\Ot$ where $\bar{K}$ is given by
$$\bar{K}= c^{-1}  \max \Big( 1 , ||a||_{\Cc^0(\Ot)} \Big)   e^{LT} .$$
\end{prop}
\begin{proof}
 Consider $(t_1,x_1)$ and $(t_2,x_2)$ in $\Ot$. Let us also suppose that $e(t_1,x_1)>e(t_2,x_2)$, the other case being symmetrical.
We infer from Proposition \ref{prop:Lip} that
\begin{equation}
 |\phi(e(t_1,x_1),t_1,x_1)-\phi(e(t_1,x_1),t_2,x_2)| \leq \max\Big(1,||a||_{\Cc^0(\Ot)}\Big)e^{LT}(|t_1-t_2|+|x_1-x_2|).
\end{equation}
Since $e(t_1,x_1)>0$, we have that $\phi ( e(t_1,x_1), t_1,x_1)=0$, and  
\begin{equation}
 \phi(e(t_1,x_1),t_2,x_2)\geq c (e(t_1,x_1)-e(t_2,x_2)) \geq 0.
\end{equation}
Therefore we end up with
\begin{equation}
 |e(t_1,x_1)-e(t_2,x_2)|\leq \bar{K}(|t_1-t_2|+|x_1-x_2|).
\end{equation}
\end{proof}

\subsection{Strong solutions} \label{subsec:strong}

Let $a\in\COI$, $y_l\in\mathcal{C}^1([0,T])$, and 
$y_0\in\mathcal{C}^1([0,1])$ be given, and assume that the following compatibility conditions hold:
\begin{equation}\label{eq:compCond}
 y_l(0)=y_0(0),\qquad y_l'(0)+a(0,0)y_0'(0)=0.
\end{equation}
We consider the following boundary initial value problem: 
\ba
&&\partial _t y + a(t,x) \partial _x y =0,\quad  (t,x)\in (0,T)\times (0,1), \label{B1}\\
&&y(t,0)= y_l(t), \quad t\in (0,T),\label{B2}\\
&&y(0,x) = y_0(x), \quad x\in (0,1). \label{B3} 
\ea
A {\em strong solution} of \eqref{B1}-\eqref{B3} is any function $y\in \Cc ^1([0,T]\times [0,1])$ such that 
\eqref{B1}-\eqref{B3} hold pointwise. 

We define a function $y:\Ot \to \R$ in the following way:
\begin{equation}\label{eq:FORM}
y(t,x)=
\begin{cases}
y_l(e(t,x)) & \text{ if } (t,x)\in J,\\ 
y_0(\phi(0,t,x)) & \text{ if } (t,x)\in I\cup P.
\end{cases}
\end{equation}
\begin{prop}
\label{prop25}
Let $y$ be as in \eqref{eq:FORM}. Then $y$ is a strong solution of \eqref{B1}-\eqref{B3}. 
Besides, we have the estimates
\begin{gather}
 ||y||_{\mathcal{C}^0(\Ot)}\leq \max\left(||y_0||_{\mathcal{C}^0([0,1])},||y_l||_{\mathcal{C}^0([0,T])}\right), \label{eq:ES}\\
||\nabla y||_{\mathcal{C}^0(\Ot)}\leq \max\left(||y_0'||_{\mathcal{C}^0([0,1])},||y_l'||_{\mathcal{C}^0([0,T])}\frac{||a||_{\mathcal{C}^0(\Ot)}}{c}\right)\exp\left(T||\partial_x a||_{\mathcal{C}^0(\Ot)}\right).
\end{gather}
\end{prop}
\begin{proof}
One can see that  $y$ is of class $\Cc^1$ on $I$ and $J$,
with the derivatives given by:
\begin{eqnarray*}
&&\partial_t y(t,x)=y_l'(e(t,x)) \frac{a(t,x)}{a(e(t,x),0)}\exp\left(-\int_{e(t,x)}^t{\!\!\!\!\!\!\!\partial_x a(s,\phi(s,t,x))ds}\right),\qquad \forall (t,x)\in J,\\
&&\partial_x y(t,x)=-y_l'(e(t,x)) \frac{1}{a(e(t,x),0)}\exp\left(-\int_{e(t,x)}^t{\!\!\!\!\!\!\!\partial_x a(s,\phi(s,t,x))ds}\right), \qquad \forall (t,x)\in J,\\
&&\partial_t y(t,x)=-y_0'(\phi(0,t,x))a(t,x)\exp\left(-\int_0^t{\partial_x a(s,\phi(s,t,x))ds}\right), \qquad \forall (t,x)\in I,\\
&&\partial_x y(t,x)=y_0'(\phi(0,t,x))\exp\left(-\int_0^t{\partial_x a(s,\phi(s,t,x))ds}\right), \qquad \forall (t,x)\in I. 
\end{eqnarray*}
It follows from the first equation in \eqref{eq:compCond} and the continuity of $e$ that $y$ is continuous at each point of $P$. 
Note that $y$ is differentiable in directions $t$ and $x$ in the following way: for all $t\in (0,f(0,0))$
\begin{gather*}
\underset{h\ra 0^+}{\lim}\frac{y(t+h,\phi(t,0,0))}{h}=y_l'(0) \frac{a(t,\phi(t,0,0))}{a(0,0)}\exp\left(-\int_{0}^t{\partial_x a(s,\phi(s,0,0))ds}\right),\\
\underset{h\ra 0^-}{\lim}\frac{y(t+h,\phi(t,0,0))}{h}=-y_0'(0)a(t,\phi(t,0,0))\exp\left(-\int_0^t{\partial_x a(s,\phi(s,0,0))ds}\right),\\
\underset{h\ra 0^+}{\lim}\frac{y(t,\phi(t,0,0)+h)}{h}=y_0'(0)\exp\left(-\int_0^t{\partial_x a(s,\phi(s,0,0))ds}\right),\\
\underset{h\ra 0^-}{\lim}\frac{y(t,\phi(t,0,0)+h)}{h}=-y_l'(0) \frac{1}{a(0,0)}\exp\left(-\int_{0}^t{\partial_x a(s,\phi(s,0,0))ds}\right).
\end{gather*}
Using the second equation in \eqref{eq:compCond}, we see that $y\in \Cc^1(\Ot)$. The fact that $y$ satisfies \eqref{eq:TR} follows from a straightforward calculation.
\end{proof}

\subsection{Weak solutions}\label{subsec:weak}
Now we consider the case when $a\in \Cc^0(\Ot)$ and $\partial_x a \in L^\infty(\ot)$. We still assume that 
\begin{equation}
a(t,x)\geq c>0, \quad \forall (t,x)\in \Ot .
\end{equation}
  We begin by introducing the space:
\begin{equation}
\mathcal{T}=\{\psi \in \Cc^1(\Ot);\ \psi(t,1)=\psi(T,x)=0\quad   \forall (t,x)\in \Ot\}.
\end{equation}
We say that a function $y\in L^1(\ot)$ is a {\em weak solution} of \eqref{B1}-\eqref{B3} 
if for any $\psi \in \mathcal{T}$ we have
\begin{multline}\label{eq:intsol}
 \iint_{\ot}{\!\!\!\!\!\!\!\!\!\!\!\!\!\!y(t,x)(\psi_t(t,x)+a(t,x) \psi_x(t,x)+a_x(t,x) \psi(t,x))dt dx}\\
 +\int_0^T{\!\!\!\psi(t,0)y_l(t) a(t,0)dt}+\int_0^1{\!\!\!\psi(0,x)y_0(x)dx}=0.
\end{multline}
Using the results of Section \ref{subsec:strong}, it is clear that a strong solution is also a weak solution. Conversely, any weak solution which is in $C^1([0,T]\times [0,1])$ 
 is a strong solution. Note that the definition of weak solution makes sense for $y_l\in L^1(0,T)$ and $y_0\in L^1(0,1)$.
\begin{prop}\label{prop:weak}
 Let us suppose that $a$, $y_l$ and $y_0$ are uniformly Lipschitz continuous 
 with Lipschitz constants $L$, $L_l$ and $L_0$, respectively,  and that $y_l(0)=y_0(0)$. Then the function $y$ defined by
\begin{equation}\label{eq:FORMbis}
y(t,x)=
\begin{cases}
y_l(e(t,x)) & \text{ if } (t,x)\in J,\\ 
y_0(\phi(0,t,x)) & \text{ if } (t,x)\in I\cup P,
\end{cases}
\end{equation}
is a weak solution of \eqref{B1}-\eqref{B3}. Furthermore, $y$ is $M$-Lipschitz continuous on $[0,T] \times [0,1] $ with $M$ defined by
\begin{equation}\label{eq:BornLip}
M:=\max(\frac{L_l}{c},L_0)\max\Big(1,||a||_{\Cc^0(\Ot)}\Big)e^{LT}.
\end{equation}
Finally, $y$ is the unique solution in the class $\text{Lip}([0,T]\times [0,1])$ of system \eqref{B1}-\eqref{B3},
with \eqref{B1} understood in the distributional sense, and \eqref{B2}-\eqref{B3}  pointwise.
\end{prop}
\begin{proof}
Using standard regularization arguments, it is possible to find $a^n\in \Cc^1(\Ot)$, $y_0^n \in \Cc^1([0,1])$ and $y_l^n\in \Cc^1 ([0,1])$ such that:
\begin{gather}
\forall \epsilon \in (0,\min(T,1))\quad 
||a^n-a||_{\Cc^0(\Ot)}+||y_0^n-y_0||_{\Cc^0([\epsilon,1])}+||y_l^n-y_l||_{\Cc^0([\epsilon,T])}\underset{n\ra +\infty}{\ra} 0 ,\\
||a_n||_{\Cc^0(\Ot)}\leq  2 ||a||_{\Cc^0(\Ot)},\quad ||\partial_x a_n ||_{L^\infty(\ot)}\leq 2 ||\partial_x a||_{L^\infty(\ot)},\\
||y_0^n||_{\Cc^0([0,1])}\leq 2 ||y_0||_{\Cc^0([0,1])},\quad ||{y_0^n}'||_{L^\infty(0,1)}\leq  2 ||y_0'||_{L^\infty(0,1)},\\
||y_l^n||_{\Cc^0([0,T])}\leq 2 ||y_l||_{\Cc^0([0,T])},\quad ||{y^n_l}' ||_{L^\infty(0,T)}\leq 2 ||y_l'||_{L^\infty(0,T)},\\
y_0^n(0)=y_l^n(0),\ \ {y_0^n}'(0)={y_l^n}'(0)=0 \qquad \forall n\in\N .
\end{gather}
Using Proposition  \ref{prop25} we infer the existence of a strong solution $y^n\in \Cc ^1( [0,T]\times [0,1] )$ of
\begin{equation}
\begin{cases}
\partial_t y^n +a^n \partial_x y^n=0,\\
y^n(t,0)=y^n_l(t),\qquad y^n(0,x)=y^n_0(x),  
 \end{cases}\forall (t,x)\in \ot.
\end{equation}
$y^n$ is given by \eqref{eq:FORM}, with $y_l,y_0,e$ and $\phi$ replaced by $y_l^n,y_0^n,e^n$ and $\phi ^n$, respectively.
Note that $(t,x)\in I^n$ (resp. $(t,x)\in J^n$) for $n$ large enough if $(t,x)\in I$ (resp. $(t,x)\in J$). 
Using Proposition  
\ref{prop:regu}, \eqref{eq:Gron} and \eqref{eq:FORM}, we see that
\begin{equation}
y^n(t,x) \underset{n \ra +\infty}{\ra} y(t,x),\qquad \forall (t,x)\in I\cup J.
\end{equation}
Note that $P=\Ot\setminus (I\cup J)$ has zero Lebesgue measure. 
An application of the dominated convergence theorem yields
$$||y^n-y||_{L^1(\ot)}\underset{n\ra +\infty}{\ra}0.$$
Using the other convergence assumptions about $a^n$, $y^n_l$, and $y^n_0$, we can pass to the limit in \eqref{eq:intsol}. This shows that 
$y$ is a weak solution of \eqref{eq:TR}.\par
To prove the regularity of $y$ we distinguish two cases.\\
Assume first that  both $(t_1,x_1)$ and $(t_2,x_2)$ are in $J\cup P$.  Using \eqref{eq:FORM} and Proposition \ref{prop:regBis}, we have that
\begin{align*}
|y(t_1,x_1)-y(t_2,x_2)|&=|y_l(e(t_1,x_1))-y_l(e(t_2,x_2))|\\
&\leq L_l |e(t_1,x_1)-e(t_2,x_2))|\\
& \leq \frac{L_l}{c} \max\Big(1,||a||_{\Cc^0(\Ot)}\Big)e^{LT}(|t_1-t_2|+|x_1-x_2|).
\end{align*}
Next, if we assume that $(t_1,x_1)$ and $(t_2,x_2)$ are in $I\cup P$, then we can use \eqref{eq:FORM} and Proposition \ref{prop:Lip} to obtain that
\begin{align*}
 |y(t_1,x_1)-y(t_2,x_2)|&=|y_0(\phi(0,t_1,x_1))-y_0(\phi(0,t_2,x_2))|\\
& \leq L_0 |\phi(0,t_1,x_1)-\phi(0,t_2,x_2)|\\
& \leq L_0 \max\Big(1,||a||_{\Cc^0(\Ot)}\Big)e^{LT}(|t_1-t_2|+|x_1-x_2|).
\end{align*}
Finally, if $(t_1,x_1)\in J$ and $(t_2,x_2)\in I$,
 we consider an intermediate point on $P$ belonging to the boundary of the rectangle $[  \min (t_1,t_2),\max (t_1,t_2) ] \times  [\min (x_1,x_2),\max (x_1,x_2) ] $ and use the 
 estimates  above.
 
Let us now check that $y$ is the only solution to \eqref{B1}-\eqref{B3} in the class $\text{Lip} (\Ot )$. First, 
picking any $\psi \in \Cc _0^\infty ((0,T)\times (0,1))$ in
\eqref{eq:intsol}, we see that \eqref{B1} holds in ${\mathcal D}'((0,T)\times (0,1))$. Note that each term in \eqref{B1} belongs to 
$L^\infty (\ot )$, so that \eqref{B1} holds also pointwise a.e. Scaling in \eqref{B1} by $\psi \in {\mathcal T}$ and comparing to 
\eqref{eq:intsol}, we obtain that \eqref{B2} and \eqref{B3} hold a.e., and also everywhere by continuity of $y_l$, $y_0$, and $y$.
Thus $y$ solves \eqref{B1}-\eqref{B3}. If $\tilde y  \in \text{Lip} (\Ot )$ is another solution of \eqref{B1}-\eqref{B3}, then 
$\hat y := y-\tilde y \in \text{Lip} (\Ot )$ solves
\ba
\partial _t \hat y + a(t,x) \partial _x \hat y &=& 0  \quad \text{ in } {\mathcal D}' ((0,T)\times (0,1)), \label{D1}\\
\hat y(t,0) &=& 0 \quad \text{ in } (0,T), \label{D2} \\
\hat y(0,x) &=& 0 \quad \text{ in } (0,1). \label{D3}   
\ea
Scaling in \eqref{D1} by $2\hat y$, integrating by parts and using  \eqref{A4}, \eqref{D2}, and \eqref{D3}, we obtain
\[
||\hat y(t)||_2^2 = \int_0^t \!\!\! \int_0^1 (\partial _x a) | \hat y|^2 dxds - \int_0^t a(t,1) | \hat y (t,1)|^2 dt  \le L\int_0^t ||\hat y (s) ||_2 ^2 ds. 
\]
This yields $\hat y \equiv 0$ by Gronwall's lemma.  The proof of Proposition \ref{prop:weak} is complete.
\end{proof}

\section{Finite-time boundary stabilization of a system of two conservation laws}
In this section, we consider the system
\begin{equation}\label{eq:sys}
 \begin{cases}
  \partial_t u+\lambda(u,v) \partial_x u=0,\\
\partial_t v +\mu(u,v) \partial_x v=0,
 \end{cases}
(t,x)\in (0,+\infty)\times (0,1).
\end{equation}
where $\lambda$ and $\mu$ are given functions with
\begin{eqnarray}
&&\lambda,\mu \in \Cc ^\infty (\R ^2,\R), \label{A1}\\
&&\mu(u,v)\leq -c<0<c\leq \lambda(u,v), \qquad \forall (u,v)\in \R^2  \label{A2}
\end{eqnarray}
for some constant $c>0$. 
We aim to prescribe a control in a feedback form on the boundary conditions $u(t,0)$ and $v(t,1)$ 
so that for some time $T$ we have for any small (in $\text{\rm Lip} ([0,1])$) initial data $u_0$ and $v_0$
\begin{equation}\label{eq:finStab}
u(T,x)=v(T,x)=0,\qquad \forall x\in (0,1).
\end{equation}
\begin{Rk}
\begin{enumerate}
\item If we intend to stabilize the system around a non null (but constant) equilibrium state $(\bar{u},\bar{v})\in \R ^2$, it is sufficient to consider the new unknowns 
$\tilde u:= u-\bar{u}$, $\tilde v := v-\bar{v}$ that satisfy a system similar to \eqref{eq:sys}, and to stabilize $(\tilde{u},\tilde{v} )$ around $(0,0)$.
\item  Note that, since we are only interested in proving a local stabilization result, the condition \eqref{A2} is not too much restrictive. It should be seen as 
$\lambda(\bar{u},\bar{v}) >0$ and $\mu(\bar{u},\bar{v}) < 0$.
\end{enumerate}
\end{Rk}
After introducing the boundary feedback law, we will show the existence and uniqueness of the solution to the closed loop system and check that
the property \eqref{eq:finStab} indeed holds for this choice of feedback law. 

We now come back to the quasilinear system \eqref{eq:sys} that we complete as follows:
\ba
&&\label{eq:closedSys}
 \begin{cases} 
  \partial_t u+\lambda(u,v) \partial_x u=0, \qquad (t,x)\in (0,+\infty)\times (0,1),\\
\partial_t v +\mu(u,v) \partial_x v=0, \qquad (t,x)\in (0,+\infty)\times (0,1),\\
 \end{cases} \\
 \label{C1}
&&\begin{cases}
\displaystyle \frac{\mathrm{d}}{\mathrm{d}t}u(t,0)=-K\mathrm{sgn}(u(t,0))|u(t,0)|^\gamma,\qquad  t > 0,\\[3mm] 
 \displaystyle    \frac{\mathrm{d}}{\mathrm{d}t}v(t,1)=-K\mathrm{sgn}(v(t,1))|v(t,1)|^\gamma, \qquad t >  0,
\end{cases} \\
&&\ \ u(0,x)=u_0(x),\ \  v(0,x) =v_0(x),\qquad  x\in (0,1) \label{CC1}
\ea
with $(K,\gamma)\in (0,+\infty)\times (0,1)$ arbitrarily chosen. We aim to use Schauder fixed-point theorem to prove the local in time existence
of solutions $(u,v)$ of \eqref{eq:closedSys}-\eqref{CC1} in some class of Lipschitz continuous functions. 
By {\em  solution}, we mean that \eqref{eq:closedSys} is satisfied in the distributional sense, and that \eqref{C1}-\eqref{CC1} are satisfied pointwise.
Actually, we shall use the results of the previous section and define
$u$ as the weak solution of the transport equation \eqref{B1}-\eqref{B3} with $a(t,x)= \lambda (\tilde u(t,x),\tilde v(t,x))$ for some given pair $(\tilde u,\tilde v)$ in the same class,
$y_l(t)=u_l(t)$ (see below \eqref{eq:fp2}), and $y_0(x)=u_0(x)$, and similarly for $v$. 

\subsection{Notations}
Let $C_1>0$ and $C_2>0$ be given, and pick any $u_0,v_0\in \Lip([0,1])$ with
\begin{gather}
\max(||u_0||_\infty,||v_0||_\infty) \le C_1,\label{eq:BornInit}\\
\max(||u_0'||_\infty,||v_0'||_\infty)\le C_2 .\label{eq:BornInitGrad}
\end{gather}
Let 
\[
T:=\frac{1}{c}+\frac{C_1^{1-\gamma}}{(1-\gamma ) K} \cdot
\]
We define $u_l$ and $v_r$ as the solutions of the following ODEs
\begin{equation}\label{eq:fp2}
   \begin{cases}
    \displaystyle \frac{\mathrm{d}}{\mathrm{d}t}u_l(t)=-K\mathrm{sgn}(u_l(t))|u_l(t))|^\gamma,\quad u_l(0)=u_0(0),\\[3mm] 
    \displaystyle \frac{\mathrm{d}}{\mathrm{d}t}v_r(t)=-K\mathrm{sgn}(v_r(t))|v_r(t)|^\gamma,\quad v_r(0)=v_0(1).\\
   \end{cases}
\end{equation} 
 An obvious calculation gives 
 \begin{eqnarray*}
 u_l(t) = \left\{
 \begin{array}{ll}
 \text{sgn} (u_0(0)) \left( |u_0(0)| ^{1-\gamma} -(1-\gamma ) K t \right) ^{\frac{1}{1-\gamma }} \quad &\text{ if } 0\le t\le \frac{|u_0(0) |^{1-\gamma}}{(1-\gamma ) K} ,\\
 0 & \text{ if } t\ge \frac{|u_0(0) |^{1-\gamma}}{(1-\gamma ) K} ,
 \end{array} 
 \right.
 \end{eqnarray*} 
 and 
  \begin{eqnarray*}
 v_r(t) = \left\{
 \begin{array}{ll}
 \text{sgn} (v_0(1)) \left( |v_0(1)| ^{1-\gamma} -(1-\gamma ) K t \right) ^{\frac{1}{1-\gamma }} \quad &\text{ if } 0\le t\le \frac{|v_0(1) |^{1-\gamma}}{(1-\gamma ) K} ,\\
 0 & \text{ if } t\ge \frac{|v_0(1) |^{1-\gamma}}{(1-\gamma ) K} ,
 \end{array} 
 \right.
 \end{eqnarray*} 
 Clearly
\begin{align}
   &\forall t\geq T-\frac{1}{c},\quad v_r(t)=u_l(t)=0,\\
   &\max(||u_l||_\infty,||v_r||_\infty)\leq C_1,\label{eq:BornGauche}\\
   &\max(||u_l'||_\infty,||v_r'||_\infty)\leq K C_1^\gamma \label{eq:BornGaucheGrad}.
\end{align}
Let us also introduce
\begin{align}
M_1&:=\max\Big(||\lambda||_{\Cc^0([-C_1,C_1]^2)},||\mu||_{\Cc^0([-C_1,C_1]^2)}\Big),\label{eq:BornVP}\\
M_2&:=\max\Big(||\partial_u \mu||_{\Cc^0([-C_1,C_1]^2)},||\partial_v \mu||_{\Cc^0([-C_1,C_1]^2)},||\partial_u \lambda||_{\Cc^0([-C_1,C_1]^2)},
||\partial_v \lambda||_{\Cc ^0([-C_1,C_1]^2)}\Big)\label{eq:BornVPGrad}.
\end{align}
Let us pick a positive number $C_3$. Let  $\mathcal{D}$ denote the domain
\begin{equation}\label{eq:DefD}
   \mathcal{D}:=\Big\{(u,v)\in \Lip([0,T]\times [0,1])^2; \ \max(||u||_\infty,||v||_\infty)\leq C_1, \text{ and } u \text{ and }v\text{ are }C_3\text{-Lipschitz} \Big\}.
\end{equation} 
Let us equip the domain $\mathcal{D}$ with the topology of the uniform convergence. Then,  by Ascoli-Arzela theorem, $\mathcal D$
 is a  compact set in $\Cc^0(\Ot)^2$. 

The main result in this section is the following
\begin{thm}
\label{main1} Assume that  $C_1>0$ and $C_2>0$ are such that 
\be 
\label{W1}
T M_2 \max\Big(1,M_1\Big)\max\Big(\frac{K C_1^\gamma}{c}, C_2\Big)\leq \frac{1}{2 e}
\ee
and let $C_3= (2T M_2)^{-1}$. Pick any pair $(u_0,v_0)\in \Lip ([0,1])^2$ satisfying \eqref{eq:BornInit}-\eqref{eq:BornInitGrad}. Then 
there exists a unique solution $(u,v)$ of \eqref{eq:closedSys}-\eqref{CC1} in the class $\mathcal{D}$. Furthermore, the solution is global in time
with $u(t,.)=v(t,.)=0$ for $t\ge T$. Finally, the equilibrium state $(0,0)$ is stable in $\text{Lip}([0,1])^2$ for \eqref{eq:closedSys}-\eqref{CC1}; that is
\be
\label{stab1}
||(u,v)||_{L^\infty (\R ^+; \text{\rm Lip}([0,1])^2)} \to 0 \ \ \text{ as } \ \ ||(u_0,v_0)||_{\text{\rm Lip}([0,1])^2}\to 0. 
\ee
\end{thm}
The first task consists in constructing a solution of the closed loop system as a fixed point of a certain operator.
\subsection{Definition of the operator}
If $(\tu,\tv)\in \mathcal{D}$ are given,  we define $(u,v)=\mathcal{F}(\tu,\tv)$ as follows:
the function $u$ is the weak solution of the system
\begin{equation}\label{eq:TransDroit}
\begin{cases}
 \partial_t u+\lambda(\tu,\tv)\partial_x u=0,\\
u(t,0)=u_l(t),\qquad u(0,x)=u_0(x),
\end{cases}
\forall (t,x)\in [0,T]\times [0,1],
\end{equation}
and the function $v$ is the weak solution of the system
\begin{equation}\label{eq:TransGauche}
\begin{cases}
 \partial_t v+\mu(\tu,\tv)\partial_x v=0,\\
v(t,1)=v_r(t),\qquad v(0,x)=v_0(x),
\end{cases}
\forall (t,x)\in [0,T]\times [0,1].
\end{equation}

\subsection{Stability of the domain}
In this part, we show that for a certain choice of $C_1,\ C_2,\ C_3$, we have 
$$\mathcal{F}(\mathcal{D})\subset \mathcal{D}.$$ 
We first apply the results of Section 2 to get the following
\begin{lem}
\label{lem1}
Let $C_1,C_2,C_3$ be any positive numbers, and let $u_0,v_0\in Lip ([0,1])$ satisfying   \eqref{eq:BornInit}-\eqref{eq:BornInitGrad}.
For given $(\tu,\tv)\in \mathcal{D}$, let  $(u,v)=\mathcal{F}(\tu,\tv)$. Then the functions $u$ and $v$ are Lipschitz continuous on $[0,T]\times [0,1]$ and they 
satisfy the following estimates
\begin{gather}
\max(||u||_{\Cc^0(\Ot)},||v||_{\Cc^0(\Ot)})\leq C_1,\label{eq:direct}\\
\max(||\partial_x u||_{L^\infty(\ot)},||\partial_x v||_{L^\infty(\ot)},||\partial_t u||_{L^\infty(\ot)},||\partial_t v||_{L^\infty(\ot)})\hspace{1.4cm}\notag\\
\hspace{7cm} \leq \max\Big(\frac{K C_1^\gamma}{c}, C_2\Big)\max\Big(1,M_1\Big) \exp\Big(2T M_2 C_3  \Big).\label{eq:grad} 
\end{gather}
\end{lem}
\begin{proof}
Estimate \eqref{eq:direct} follows directly from \eqref{eq:TransDroit}, \eqref{eq:TransGauche}, \eqref{eq:FORMbis}, \eqref{eq:BornInit} and \eqref{eq:BornGauche}.\par
Estimate \eqref{eq:grad} can be deduced applying \eqref{eq:BornLip} for \eqref{eq:TransDroit} and \eqref{eq:TransGauche}, and using 
\eqref{eq:BornInitGrad}, \eqref{eq:BornGaucheGrad}, \eqref{eq:BornVP}, \eqref{eq:BornVPGrad} and \eqref{eq:DefD}.
\end{proof}
Thanks to Lemma \ref{lem1},  we see that the domain $\mathcal D $ is stable by $\mathcal F$ as soon as
\begin{equation}\label{eq:StabCond}
 \max\Big(\frac{K C_1^\gamma}{c}, C_2\Big)\max\Big(1,M_1\Big) \exp\Big(2T M_2 C_3  \Big)\leq C_3.
\end{equation}
This can be written as
\begin{equation}
\label{C2}
\max\Big(\frac{K C_1^\gamma}{c}, C_2\Big)\leq \frac{C_3 \exp\Big(-2T M_2 C_3  \Big)}{\max\Big(1,M_1\Big)}.
\end{equation}
For given $C_1$ and $C_2$, $T,M_1$ and $M_2$ are fixed. Note that $T$, $M_1$ and $M_2$ are independent of $C_2$, and that they are 
nondecreasing in $C_1$.
Therefore, as a function of $C_3$ the supremum of the right-hand side of \eqref{C2} is attained for $C_3= (2T M_2)^{-1}$, and 
for this value of $C_3$ the condition on $C_1$ and $C_2$ for the domain to be stable reads
\be T M_2 \max\Big(1,M_1\Big)\max\Big(\frac{K C_1^\gamma}{c}, C_2\Big)\leq \frac{1}{2 e}.
\label{condC1C2}
\ee
But the term in the left-hand side  of \eqref{condC1C2} 
tends to $0$ when $C_1$ and $C_2$ tend to $0$, so that for $C_1,C_2$ small enough the condition 
\eqref{condC1C2} is satisfied and $\mathcal{D}$ is stable by $\mathcal{F}$. 

\subsection{Continuity of the operator}
In this part we consider a sequence $\{(\tu_n,\tv_n)\} \subset \mathcal{D}$ and a couple $(\tu,\tv) \in \mathcal{D}$ such that
\begin{equation}
 \max\Big( ||\tu_n-\tu||_{\Cc^0(\Ot)},||\tv_n-\tv||_{\Cc^0(\Ot)}\Big)\underset{n\ra +\infty}{\ra}0.
\end{equation}
 Let us now define 
\begin{equation}
(u_n,v_n)=\mathcal{F}(\tu_n,\tv_n) \quad \text{ for }  n\geq 0,
 \quad \text{ and }\quad (u,v)=\mathcal{F}(\tu,\tv).
 \end{equation}
Our goal in this subsection is to show that 
\begin{equation}
 \max\Big( ||u_n-u||_{\Cc^0(\Ot)},||v_n-v||_{\Cc^0(\Ot)}\Big)\underset{n\ra +\infty}{\ra}0.
\end{equation}
We need the following 
\begin{lem}\label{lem:ConvFaible}
For almost all $(t,x)\in \Ot$, we have
\begin{equation}
(u_n(t,x),v_n(t,x)) \underset{n\ra +\infty}{\ra}(u(t,x),v(t,x)).
\end{equation}
\end{lem}
\begin{proof}
Let us show that $u_n(t,x)\ra u(t,x)$, the convergence $v_n(t,x) \ra v(t,x) $ being similar.\par
The fact that $(\tu_n,\tv_n)$ converges uniformly toward $(\tu,\tv)$ on $\Ot$ implies that $\lambda(\tu_n,\tv_n)$ converges 
uniformly toward $\lambda(\tu,\tv)$ on $\Ot$. Furthermore, since $(\tu_n,\tv_n)\in \mathcal{D}$ for all $n$, we see that the 
functions $\lambda(\tu_n,\tv_n)$ are uniformly Lipschitz continuous for $n\ge 0$. This will allow us to use Proposition 
 \ref{prop:regu}.  
To this end, we consider the flow $\phi_n$ (resp. $\phi$) of $\lambda(\tu_n,\tv_n)$ (resp. $\lambda(\tu,\tv)$). In the 
same way, we define $e_n$ and $e$, $I_n$ and $I$, $J_n$ and $J$, $P_n$ and $P$. Using \eqref{eq:FORMbis} we have that
\begin{equation}
u_n(t,x)=
\begin{cases}
u_l(e_n(t,x)) & \text{ if } (t,x)\in J_n,\\ 
u_0(\phi_n(0,t,x)) & \text{ if } (t,x)\in I_n\cup P_n,
\end{cases}
\end{equation}
and also
\begin{equation}
u(t,x)=
\begin{cases}
u_l(e(t,x)) & \text{ if } (t,x)\in J,\\ 
u_0(\phi(0,t,x)) & \text{ if } (t,x)\in I\cup P.
\end{cases}
\end{equation}
We infer from Proposition \ref{prop:regu} that
\begin{equation}
 e_n(t,x) \underset{n\ra +\infty}{\ra} e(t,x),
 \quad \forall (t,x)\in \Ot .
\end{equation}
This shows in particular that if $(t,x)\in J$, then $e(t,x)>0$ and hence $e_n(t,x)>0$ for $n$ large enough, i.e. $(t,x)\in J_n$ for $n$ large enough. 
Therefore
$$u_n(t,x)\underset{n\ra +\infty}{\ra}u(t,x),\qquad \forall (t,x)\in J.$$
Now if $(t,x)\in I$, then $e(t,x)=0$ and $\phi(0,t,x)>0$. Since $\lambda \geq c>0$, this implies the existence of $\epsilon>0$ such that
\begin{equation}
\epsilon <\phi(s,t,x) ,\qquad  
\forall s\in [0,t]. 
\end{equation}
Combined with \eqref{eq:Gron},
this shows that for $n$ large enough $e_n(t,x)=0$ and $\phi_n(0,t,x)\ra \phi(0,t,x)$, so we conclude that
\begin{equation}
 u_n(t,x)=u_0(\phi_n(0,t,x))\underset{n\ra +\infty}{\ra }u_0(\phi(0,t,x))=u(t,x).
\end{equation}
Finally, $P$ is clearly negligible and $I\cup P \cup L= [0,T]\times [0,1].$ 
\end{proof}

To strengthen this convergence, we just need to recall that for every $n\geq 0$, we have $(u_n,v_n)\in \mathcal{D}$ 
which is compact in $\Cc^0(\Ot)$. According to Lemma \ref{lem:ConvFaible}, 
the only possible limit point is $(u,v)$ and therefore we get the convergence of the whole sequence in $\mathcal{D}$; that is, 
 \begin{equation}
 \max\Big( ||u_n-u||_{\Cc^0(\Ot)},||v_n-v||_{\Cc^0(\Ot)}\Big)\underset{n\ra +\infty}{\ra}0.
\end{equation}
This shows that the operator $\mathcal{F}$ is continuous on the domain $\mathcal{D}$, which is a convex compact set in $\Cc^0(\Ot)^2$. It follows then 
from Schauder fixed-point theorem that $\mathcal{F}$ has a fixed-point. This proves the existence of solutions on the time interval $[0,T]$. 

\subsection{Uniqueness of the solution}
Let $u_0,v_0\in \Lip([0,1])$ be as in \eqref{eq:BornInit}-\eqref{eq:BornInitGrad}. Assume given two pairs  $(u^1,v^1),(u^2,v^2)\in {\mathcal D} $ of solutions of 
\eqref{eq:closedSys}-\eqref{CC1}; that is, if $u_l$ and $v_r$ are defined as in \eqref{eq:fp2}, then $u^i$, $i=1,2$, is  a (weak) solution of 
\[
\left\{
\begin{array}{l}
\partial_t u^i+\lambda(u^i,v^i) \partial_x u^i=0, \\
u^i(t,0)=u_l(t),\quad u^i(0,x)= u_0(x),
\end{array}
\right.
\qquad (t,x)\in (0,T)\times (0,1),
\]
while $v^i$, $i=1,2$, is a (weak) solution of 
\[
\left\{
\begin{array}{l}
\partial_t v^i+\mu (u^i,v^i) \partial_x v^i=0, \\
v^i(t,1)=v_r(t),\quad v^i(0,x)= v_0(x),
\end{array}
\right.
\qquad (t,x)\in (0,T)\times (0,1).
\]

Let $\hat u=u^1-u^2$ and $\hat v=v^1-v^2$. Note that $\hat u,\hat v\in \text{Lip}([0,T]\times [0,1])=W^{1,\infty} ((0,T)\times (0,1))$ and that $\hat u,\hat v$ fulfill 
\begin{eqnarray}
&&\partial_t \hat u+\lambda ^1 \partial_x \hat u + \hat \lambda \partial _x  u^2=0,\qquad (t,x)\in (0,T)\times (0,1), \label{AA1}\\
&&\partial_t \hat v+\mu ^1 \partial_x \hat v + \hat \mu \partial _x  v^2=0,\qquad (t,x)\in (0,T)\times (0,1), \label{AA2}\\
&&\hat u(t,0)=\hat v(t,1)=0,\qquad \hat u(0,x)=\hat v(0,x)=0, \label{AA3}
\end{eqnarray}
where $ \lambda ^i = \lambda  ( u^i , v^i)$, $\mu ^i = \mu (u^i,v^i)$, and $\hat \lambda = \lambda ^1 - \lambda ^2$,
$\hat \mu =\mu ^1 - \mu ^2$.

Multiplying in \eqref{AA1}  by $ 2\hat u$, in \eqref{AA2} by $ 2 \hat v$, integrating over $(0,t)\times (0,1)$, and adding the two equations 
gives
\[
||\hat u(t)||_2^2 + ||\hat v(t)||_2^2 + 2\int _0^t \!\! \int _0^1(\lambda ^1 \hat u\partial _x \hat u + \mu ^1 \hat v\partial _x \hat v )\,  dxds
+2\int _0^t \!\! \int _0^1 (\hat \lambda \hat u\partial _x u^2 +\hat \mu  \hat v \partial _x v^2 )\, dxds =0.
\]

Using \eqref{AA3} and an integration by parts, we obtain
\begin{eqnarray*}
&&2\int _0^t \!\! \int _0^1(\lambda ^1 \hat u\partial _x \hat u + \mu ^1 \hat v\partial _x \hat v)\\
&&\qquad  = -\int _0^t \!\! \int_0^1 [ (\partial _x \lambda ^1) |\hat u|^2 + (\partial _x \mu ^1) |\hat v|^2 ] \, dxds 
+ \int_0^t [ \lambda ^1 |\hat u (s,1 ) |^2    - \mu ^1 |\hat v (s,0) |^2  ] ds \\
&&\qquad   \ge -\int_0^t\!\! \int_0^1 [(\partial _x \lambda ^1 ) |\hat u|^2 + (\partial _x \mu ^1) |\hat v|^2  ]\, dxds  
\end{eqnarray*}
where we used \eqref{A2}. On the other hand, since $\lambda$ and $\mu$ are $M_2$-Lipschitz continuous on $[-C_1,C_1]^2$, we infer that 
$\lambda ^i$ and $\mu ^i$ are $2 M_2C_3$-Lipschitz continuous on $[0,T]\times [0,1]$. In particular, 
\[
|| \partial _x \lambda ^1 ||_{\infty} \le 2 M_2C_3,\quad ||\partial _x \mu ^1 ||_\infty  \le 2 M_2 C_3
\]   
and
\begin{eqnarray*}
|\hat \lambda | &\le& M_2  (|\hat u|+ |\hat v|), \\
|\hat \mu | &\le&  M_2 (|\hat u|+ |\hat v|). 
\end{eqnarray*}
This yields
\[
|2\int _0^t \!\! \int _0^1 (\hat \lambda \hat u\partial _x u^2 +\hat \mu  \hat v \partial _x v^2 )\, dxds| \le
2M_2C_3\int_0^t\!\!\int_0^1 (|\hat u| +|\hat v |)^2dxds.
\]
We conclude that for all $t\in (0,T)$ 
\[
||\hat u(t)||_2^2 + ||\hat v(t)||_2^2  \le 6 M_2 C_3   \int_0^t (||\hat u||_2^2 +||\hat v||_2^2 ) ds.  
\]
This yields $\hat u=\hat v\equiv 0$, by Gronwall's lemma. 

\subsection{Finite-time extinction of the maximal solutions}
In this section, $(u,v)$ denotes the only solution of \eqref{eq:closedSys}-\eqref{CC1} in the class $\mathcal{D}$. 
\begin{lem}
\label{lem10}
 At time $t=T$ we have
\be
u(T,x)=v(T,x)=0,\quad \forall x\in [0,1].
\label{CBA}
\ee
\end{lem}
\noindent
{\em Proof of Lemma \ref{lem10}:}
We infer from  \eqref{C1} that
\begin{equation}\label{eq:NullCond}
u(t,0)=v(t,1)=0,\quad \forall t\geq T-\frac{1}{c}.
\end{equation}
Thanks to \eqref{A1}-\eqref{A2}, we have that
$$\lambda(u(t,x),v(t,x)) \geq c >0 >-c> \mu(u(t,x),v(t,x)),\qquad  \forall (t,x)\in \Ot .$$ 
Let $\phi^\lambda$ (resp. $\phi^\mu$) denote the flow of $\lambda(u,v)$ (resp. $\mu(u,v)$), and
let $e^\lambda$ (resp. $e^\mu$) denote the corresponding entrance times. 
(Note that  $e^\mu >0$ implies $\phi^\mu(e^\mu(t,x),t,x)=1$.) Then the following holds:
$$e^\mu(T,x)\geq T-\frac{1}{c}\quad \text{ and } \quad e^\lambda(T,x)\geq T-\frac{1}{c},
\quad \forall x\in [0,1].$$
Combining this with \eqref{eq:NullCond} and \eqref{eq:FORMbis}, we obtain \eqref{CBA}.
\qed

Finally, it is sufficient to extend $u$ and $v$ by $0$ for $t\geq T$ to get a global in time solution.
The stability property \eqref{stab1} follows at once from \eqref{eq:direct}-\eqref{eq:grad}, as the r.h.s. in \eqref{eq:direct} and \eqref{eq:grad}
tend to 0 as $(C_1,C_2)\to (0,0)$. 
The proof of Theorem \ref{main1} is complete. \qed
 
\section{Finite time stabilization with a control from one side}
In this section, we consider a system of the form 
\ba
&&\partial _t u + \lambda (u,v) \partial _x u=0, \qquad (t,x)\in (0,+\infty ) \times (0,1), \label{K1} \\
&&\partial _t v + \mu (u,v) \partial _x v = 0 , \qquad (t,x) \in (0,+\infty ) \times (0,1), \label{K2} \\
&& u(t,0)=h(v(t,0),t),\quad u(0,x)=u_0(x), \label{K3} \\
&& v(t,1) = v_r(t), \quad v(0,x) =v_0(x), \label{K4}  
\ea
where $v_r$ still solves the ODE
\be
\frac{d}{dt} v_r(t) = -K \text{sgn} (v_r(t)) |v_r(t)|^\gamma , \quad v_r(0)= v_0(1).
\label{K5}   
\ee
In \eqref{K3}, $h$ denotes some function in $\Cc ^1([-\overline{C_1},\overline{C_1}]\times \R ^+ )\cap W^{1,\infty}( (-\overline{C_1},
\overline{C_1})\times (0,+\infty ))$  for some number $\overline{C_1}>0$ such that, for some time $T_h>0$,
\be
\label{K8}
h(0,t)=0 \qquad \forall t\ge T_h.
\ee
We introduce the numbers
\begin{eqnarray*}
C_1 &\in& (0,\overline{C_1} ] ,\\ 
T&:=& \frac{1}{c}  + \max ( T_h, \frac{1}{c} + \frac{C_1^{1-\gamma}}{(1-\gamma )K} ), \\ 
C_1' &:=&  \max (C_1, ||h||_{L^\infty ((-C_1,C_1)\times (0,+\infty ))}  ) , \\
D_1 &:=& ||\partial _v h||_{L^\infty ((-C_1,C_1)\times (0,+\infty ))}, \\
D_2 &:=& ||\partial _t h||_{L^\infty ((-C_1,C_1)\times (0,+\infty ))},\\
M_1&:=&\max\Big(||\lambda||_{\Cc^0([-C_1',C_1']\times [-C_1,C_1] )},||\mu||_{\Cc^0([-C_1',C_1']\times [-C_1, C_1] )}\Big), \\
M_2&:=&\max\Big(
||\partial_u \mu||_{\Cc^0([-C_1',C_1']\times [-C_1, C_1])},
||\partial_v \mu||_{\Cc^0([-C_1',C_1']\times [-C_1, C_1])}, \\
&&\quad \qquad ||\partial_u \lambda||_{\Cc^0([-C_1',C_1']\times [-C_1, C_1] )},
||\partial_v \lambda||_{\Cc^0([-C_1',C_1']\times [-C_1, C_1] )}\Big) ,\\
C_3 &:=& \max (\frac{1}{2TM_2}   ,C_2), \\
C_3' &:=& \max (\frac{KC_1^\gamma }{c},C_2) \max (1,M_1) \exp (2TM_2C_3). 
\end{eqnarray*}

Note that, if $||v||_{\Cc ^0 ([0,T]\times [0,1] )} \le C_1$, then for all $t\in (0,T)$
\[
|u(t,0) | \le C_1' \quad \text{ and } \quad  |\partial _t u(t,0) |  \le D_1 |\partial _t v(t,0) | + D_2. 
\] 

We shall consider the following conditions
\ba
&&C_3' \le C_3, \label{K12} \\
&&C_3'' :=   \max  (  \frac{1}{c} (D_1 C_3' + D_2 ),C_2) \max (1,M_1)\exp (2TM_2C_3)  \le C_3.
\label{K13}
\ea
Note that \eqref{K12} and \eqref{K13} are satisfied if $C_1$, $C_2$, and $D_2$ are small enough.  

We introduce the set 
\begin{multline*}
{\mathcal D } := \left\{ (u,v)\in \Lip ([0,T]\times [0,1])^2; \ ||u||_{\Cc ^0 ([0,T]\times [0,1])} \le C_1',\ ||v||_{\Cc ^0 ([0,T]\times [0,1])}  \le C_1, \right.\\ 
\left. u \text{ is } C_3\text{-Lipschitz}, \  v \text{ is } C_3'\text{-Lipschitz} \right\} .
\end{multline*}

We pick a pair $(u_0,v_0)\in \Lip ([0,1])^2$ fulfilling \eqref{eq:BornInit}-\eqref{eq:BornInitGrad}
and the following compatibility condition 
\be
\label{K9}
u_0(0)=h(v_0(0),0).
\ee

Let us do some comments about the boundary condition \eqref{K3}.  
For a system of conservation laws on the interval $(0,1)$, a very general boundary condition at $x=0$ takes the form 
$f(u(t,0),v(t,0))=0$. If $\partial _u f(u_0,v_0)\ne 0$, then around $(u_0,v_0)$ an application of the Implicit
Function Theorem gives a relation of the form
\[
u(t,0)=h(v(t,0))
\] 
with $h$ a smooth function of $v$ in a neighborhood of $v_0$. Assume now that the interval represents an edge in a network, and that
the left endpoint is a multiple node (i.e. it belongs to at least two edges). The contributions of the other edges at this multiple node
can be  taken into account in $h$ through its dependence in $t$  in \eqref{K3}.

We are in a position to state the main result of this section.
 \begin{thm}
\label{main2} Assume that  $C_1,C_2$ and $D_2$ are such that the conditions \eqref{K12} and \eqref{K13} are satisfied. Then
for any pair $(u_0,v_0)\in \Lip ([0,1])^2$ fulfilling \eqref{eq:BornInit}, \eqref{eq:BornInitGrad} and \eqref{K9}, 
there exists a unique solution $(u,v)$ of \eqref{K1}-\eqref{K5} in the class $\mathcal{D}$. Furthermore, the solution is global in time
with $u(t,.)=v(t,.)=0$ for $t\ge T$. 
Finally, if $h=h(v)$, then the equilibrium state $(0,0)$ is stable in $\text{Lip}([0,1])^2$ for \eqref{K1}-\eqref{K5}; that is
\be
\label{stab2}
||(u,v)||_{L^\infty (\R ^+; \text{\rm Lip}([0,1])^2)} \to 0 \ \ \text{ as } \ \ ||(u_0,v_0)||_{\text{\rm Lip}([0,1])^2}\to 0. 
\ee
\end{thm}
\begin{proof}
It is very similar to those of Theorem \ref{main1}. If $(\tilde u,\tilde v)\in {\mathcal D}$ is given, we define $(u,v)={\mathcal F} (\tilde u ,\tilde v)$ as follows:
$u$ is the weak solution of the system
\[
\begin{cases}
 \partial_t u+\lambda(\tu,\tv)\partial_x u=0,\\
u(t,0)=h(\tv (t,0),t) ,\qquad u(0,x)=u_0(x),
\end{cases}
\forall (t,x)\in [0,T]\times [0,1],
\]
and $v$ is the weak solution of the system
\[
\begin{cases}
 \partial_t v+\mu(\tu,\tv)\partial_x v=0,\\
v(t,1)=v_r(t),\qquad v(0,x)=v_0(x),
\end{cases}
\forall (t,x)\in [0,T]\times [0,1].
\]
Then, using Proposition \ref{prop:weak} and \eqref{K12}-\eqref{K13}, one readily sees that 
\[
||u||_{\Cc ^0 ([0,T]\times [0,1])}\le C_1',\ ||v||_{\Cc  ^0([0,T]\times [0,1])} \le C_1,\  
\]
\ba
\label{ABC1}
&&u \text{ is }  C_3'' \text{-Lipschitz},  \text{ hence } u \text{ is }  C_3\text{-Lipschitz}, \\
&&v \text{ is }  C_3' \text{-Lipschitz}, \label{ABC2}
\ea
so that $\mathcal F$ maps $\mathcal D$ into itself. Let us prove that $\mathcal F$ is continuous, $\mathcal D$ being equipped with the topology of the 
uniform convergence.
Consider a sequence $\{ (\tu_n,\tv_n)\} \subset \mathcal{D}$ and a pair $(\tu,\tv) \in \mathcal{D}$ such that
\begin{equation}
 \max\Big( ||\tu_n-\tu||_{\Cc^0(\Ot)},||\tv_n-\tv||_{\Cc^0(\Ot)}\Big)\underset{n\ra +\infty}{\ra}0.
\end{equation}
 Let
\begin{equation}
(u_n,v_n)=\mathcal{F}(\tu_n,\tv_n) \quad \text{ for }  n\geq 0,
 \quad \text{ and }\quad (u,v)=\mathcal{F}(\tu,\tv).
 \end{equation}
We aim to prove that $u_n\to u$ and $v_n\to v$ uniformly on $[0,T]\times [0,1]$ as $n\to \infty$. 
We focus on $u_n$, the argument for $v_n$ being the same as those given in Lemma \ref{lem:ConvFaible}. We consider the same
$\phi _n, \phi, e_n, e, I_n, I, J_n, J,P_n$  and $P$, as in the proof of Lemma \ref{lem:ConvFaible}. Then 
\begin{equation*}
u_n(t,x)=
\begin{cases}
h (\tilde v _n (e_n(t,x),0), e_n(t,x)) & \text{ if } (t,x)\in J_n,\\ 
u_0(\phi_n(0,t,x)) & \text{ if } (t,x)\in I_n\cup P_n
\end{cases}
\end{equation*}
and
\begin{equation*}
u(t,x)=
\begin{cases}
h(\tilde v( e(t,x),0),e(t,x)) & \text{ if } (t,x)\in J,\\ 
u_0(\phi (0,t,x)) & \text{ if } (t,x)\in I\cup P.
\end{cases}
\end{equation*}
Assume first that $(t,x)\in J$. Then $e(t,x)>0$ and $e_n(t,x)>0$ for $n$ large enough, by
Proposition \ref{prop:regu}.
Since $\tilde v_n\to \tilde v$ uniformly on $[0,T]\times [0,1]$ and $e_n(t,x)\to e(t,x)$, we infer that 
\[
u_n(t,x) = h (\tilde v_n (e_n(t,x),0),e_n(t,x)) \to h (\tilde v (e(t,x),0),e(t,x)) = u(t,x).
\]
If now $(t,x)\in I$, one can repeat the argument in Lemma \ref{lem:ConvFaible}
to conclude that 
\[
u_n(t,x)=u_0(\phi _n (0,t,x))\to u_0(\phi (0,t,x)) = u(t,x). 
\]
Thus, $u_n(t,x)\to u(t,x)$ for $(t,x)\in I\cup J$, hence for a.e. $(t,x)\in [0,T]\times [0,1]$. We have also that 
$v_n(t,x)\to v(t,x)$ for a.e. $(t,x)\in [0,T]\times [0,1]$. We infer from the compactness of $\mathcal D$ in 
$\Cc ^0([0,T]\times [0,1])^2$ that $(u_n,v_n)\to (u,v)$ in  $\Cc ^0([0,T]\times [0,1])^2$. We conclude with Schauder 
fixed-point theorem  that $\mathcal F$ has a fixed-point  $(u,v)\mathcal \in \mathcal D$,  which is  a solution of \eqref{K1}-\eqref{K5} on 
$[0,T]\times [0,1]$. 

Let us now establish the uniqueness of the solution  of \eqref{K1}-\eqref{K5} in the class $\mathcal D$.  
Assume given two pairs  $(u^1,v^1),(u^2,v^2)\in {\mathcal D} $ of solutions of 
\eqref{K1}-\eqref{K5}; that is, with $v_r$ defined as in \eqref{K5}, $v^i$, $i=1,2$, is a (weak) solution of 
\[
\left\{
\begin{array}{l}
\partial_t v^i+\mu (u^i,v^i) \partial_x v^i=0, \\
v^i(t,1)=v_r(t),\quad v^i(0,x)= v_0(x),
\end{array}
\right.
\qquad (t,x)\in (0,T)\times (0,1),
\]
while $u^i$, $i=1,2$, is a (weak) solution of 
\[
\left\{
\begin{array}{l}
\partial_t u^i+\lambda(u^i,v^i) \partial_x u^i=0, \\
u^i(t,0)=h(v^i(t,0),t),\quad u^i(0,x)= u_0(x),
\end{array}
\right.
\qquad (t,x)\in (0,T)\times (0,1).
\]

Let $\hat u=u^1-u^2$ and $\hat v=v^1-v^2$. Note that $\hat u,\hat v\in W^{1,\infty} ((0,T)\times (0,1))$ and that $\hat u,\hat v$ satisfy 
\begin{eqnarray}
&&\partial_t \hat u+\lambda ^1 \partial_x \hat u + \hat \lambda \partial _x  u^2=0, \label{K21}\\
&&\partial_t \hat v+\mu ^1 \partial_x \hat v + \hat \mu \partial _x  v^2=0, \label{K22}\\
&&\hat u(t,0)=h(v^1(t,0),t) - h(v^2(t,0),t), \label{K23}\\ 
&&\hat v(t,1)=0,\label{K24}\\
&&\hat u(0,x)=\hat v(0,x)=0 \label{K25}
\end{eqnarray}
where $ \lambda ^i = \lambda  ( u^i , v^i)$, $\mu ^i = \mu (u^i,v^i)$, and $\hat \lambda = \lambda ^1 - \lambda ^2$,
$\hat \mu = \mu ^1 - \mu ^2$.

Multiplying in \eqref{K21}  by $ 2\hat u$ and  integrating over $(0,t)\times (0,1)$ gives
\ba
||\hat u(t)||^2 &=& - 2 \int _0^t \!\! \int _0^1[\lambda ^1 \hat u\partial _x \hat u + \hat \lambda \hat u\partial _x u^2] dxds \nonumber \\
 &=& \int _0^t \!\! \int _0^1[(\partial _x \lambda ^1 ) |\hat u|^2 - 2 \hat \lambda \hat u\partial _x u^2] dxds 
-\int_0^t \lambda ^1 |\hat u |^2 \big\vert _0^1ds \nonumber \\
&\leq & 2M_2 C_3 \int _0^t \!\! \int _0^1 [ |\hat u|^2 +   |\hat u | (  |\hat u| + |\hat v| ) ] dxds \nonumber \\
&&\quad + ||\lambda ||_{\Cc ^0 ( [-C_1',C_1'] \times [-C_1,C_1]) } D_1^2 \int_0^t  |\hat v (s,0) |^2 ds \label{K31}
\ea
where we used \eqref{A2}. 

Multiplying in \eqref{K22}  by $ 2\hat v$ and  integrating over $(0,t)\times (0,1)$ gives
\ba
||\hat v(t)||^2 &=& - 2 \int _0^t \!\! \int _0^1[\mu ^1 \hat v\partial _x \hat v + \hat \mu \hat v\partial _x v^2] dxds \nonumber \\
 &=& \int _0^t \!\! \int _0^1[(\partial _x \mu ^1 ) |\hat v|^2 - 2 \hat \mu \hat v\partial _x v^2] dxds 
+\int_0^t \mu ^1 |\hat v  (s,0)|^2 ds \nonumber \\
&\le& 2M_2C_3 \int_0^t\! \! \int_0^1 [ |\hat v|^2 +  |\hat v | ( |\hat u| + |\hat v|)]  dxds -c \int_0^t |\hat v (s,0)|^2 ds \label{K32}  
\ea
where we used \eqref{A2} again.
Let us introduce the energy 
\[
E(t) = ||\hat u(t)||^2 + ||\lambda ||_{ \Cc ^0 ( [-C_1',C_1']\times [-C_1,C_1])} \frac{D_1^2}{c} ||\hat v(t)||^2.
\]
Combining \eqref{K31} with \eqref{K32} yields
\[
E(t) \le C\int_0^t E(s) ds, 
\]
for some $C$ depending only on $\mathcal{D}$, so that $E\equiv 0$, by Gronwall's lemma. This proves the uniqueness. For the extinction time, we notice that from the proof of Theorem \ref{main1}
\[
v(t,x)=0, \qquad \text{ for }\   \frac{1}{c} + \frac{C_1^{1-\gamma} }{(1-\gamma )K }  \le t\le T, \ 0\le  x\le 1.
\]
Combined with \eqref{K8},  this yields 
\[
u(t,0)= h(v(t,0),t)=0, \qquad \text{ for }  \  \max \left( T_h,  \frac{1}{c} + \frac{C_1^{1-\gamma}}{(1-\gamma )K}\right) \le t\le T .
\]
Using \eqref{A2}, we conclude that 
\[
u(T,x)=0\qquad   \forall x\in [0,1].
\]
Assume now that $h=h(v)$, i.e. $D_2=0$.
The stability property \eqref{stab2} follows at once from \eqref{K8} and  \eqref{K12}-\eqref{K13}, as $C_1'\le \max (1,D_1)C_1$ and 
$(C_3',C_3'')\to (0,0)$ as $(C_1,C_2)\to (0,0)$. 
The proof of Theorem  \ref{main2} is complete.
\end{proof}
\section{Application to the regulation of water flow in channels}

In this section, we investigate the regulation of water flow in a network of open horizontal channels. We assume that the channels have a rectangular cross section
and that the friction on the walls can be neglected.  In this context, the flow of the fluid can be described in a satisfactory way by the shallow water
equations (also called Saint-Venant equations) (see \cite{HPCAB}). The control in feedback form is applied at the vertices of the network, which is assumed to be a tree.

We introduce some notations needed in what follows (we follow closely \cite{DZcras}). Let $\mathcal T$ be a tree, whose vertices (or nodes)
are numbered by the index $n\in \cN  =\{ 1,...,N\}$, and whose edges are numbered by the index 
$i\in \cI  =\{ 1,...,I \}$ with $I=N-1$. We choose a simple vertex, called the {\em root} 
of $\mathcal T$ and denoted by $\mathcal R$, and which corresponds to the index $n=N$. We choose an orientation of the edges in the tree
such that $\mathcal R$ is the ``last'' encountered vertex. It is similar to those of a fluvial network in which each edge stands for a river, and $\mathcal R$ 
indicates the place where the last river enters into the ocean. 

We denote by $l_i$ the length of the edge with index $i$. Once the orientation is chosen, each point of the $i$-th edge is identified with
a real number $x\in [0,l_i]$. The points  $x=0$ and $x=l_i$ are termed  the {\em initial point} and the {\em final point} of the $i$-edge, respectively.

Renumbering the edges if needed, we may assume that the edge with index $i$ has as initial point the vertex with the (same) index $n=i$ for all $i\in \cI$.  

We denote by $\cI _n\subset \cI$, $n=1,...,N$, the set of indices of those edges having the vertex of index $n$ as one of their ends. Let 
 \[
 \varepsilon_{i,n} = \left\{ 
 \begin{array}{ll}
 0\quad  &\text{\rm if the vertex with index $n$ is the initial point of the edge with index $i$};\\
 1\quad  &\text{\rm if the vertex with index $n$ is the final point of the edge with index $i$}.
 \end{array}
 \right.
 \]
 Note that $\varepsilon _{i,i}=0$ for all $i\in\cI $, and that $\varepsilon_{N-1,N}=1$. A node with index $n$ is said to be {\em simple} (resp. {\em multiple}) if
 $\# (\cI _n) = 1$ (resp. $\# (\cI _n)\ge 2$).   The sets of indices of simple and multiple nodes are denoted by $\cN _S$ and $\cN _M$, respectively. 
The {\em depth} of the tree is the greater number of edges in a path from one simple node to $\mathcal R$. 
\begin{figure}[http]
\begin{center}
\includegraphics[scale=0.5]{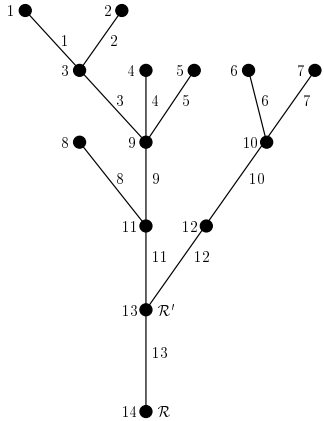}
\end{center}
\caption{ A tree with $14$ nodes, a depth equal to $5$, with simple nodes $\mathcal{N}_S=\{1,2,4,5,6,7,8,14\}$ and multiple nodes $\mathcal{N}_M=\{3,9,10,11,12,13\}$.}
\label{fig2}
\end{figure}
 
Pick any channel represented by (say) the $i$-th edge of the tree,  which is identified with the segment $[0,l_i]$. 
Then the shallow water equations read
\begin{eqnarray}
\partial _t H_i+ \partial _x (H_iV_i) &=& 0, \quad t > 0,\  0<x<l_i, \label{sw1}\\
\partial _t V_i + \partial _x( \frac{V_i^2}{2} + gH_i) &=& 0, \quad t > 0, \ 0<x<l_i, \label{sw2}
\end{eqnarray}
 where $H_i(t,x)$ (resp. $V_i(t,x)$) is the water depth (resp. the water velocity) along the $i$-th channel, and $g$ is the gravitation constant. 
 The equations \eqref{sw1}-\eqref{sw2} have to be supplemented with some initial conditions
 \be
 \label{sw3}
 H_i(0,x)=H_{i,0}(x),\quad V_i(0,x)=V_{i,0}(x), \qquad 0<x<l_i
 \ee
 and with two boundary conditions. 
 In general, there are at the two ends of the channel (i.e. at $x=0$ and at $x=l_i$) some hydraulic devices to assign the 
 values of the flow rate. Recall that the flow rate is defined along the channel as 
 \[
 Q_i(t,x):=H_i(t,x)V_i(t,x). 
 \]    
 At any multiple node $n\in \cN _M$, the equation of conservation of the flow
 \be
 \sum_{i\in \cI _n} (-1)^{\varepsilon _{i,n}} Q_i(t,\varepsilon _{i,n} l_i) =0     \label{F1}
 \ee
 has to be taken into consideration. It yields a boundary condition (coming from the physics) in which no control applies.
 Let $i_0\in \cI _n$ be the only index such that $\varepsilon _{i_0,n}=0$, namely $i_0=n$. Then \eqref{F1} can be written
 \be
 Q_{i_0} (t,0)= \sum_{i\in \cI _n, i\ne i_0} Q_i (t,l_i). \label{F2}
 \ee 
 Thus, the flow rate may be controlled at the final points of the edges of indices $i\ne i_0$, while it is prescribed by \eqref{F2}
 at the initial point of the edge of index $i_0$. 
 
 We aim to stabilize the system around some equilibrium state, represented by a sequence $\{ (H_i^*,V_i^*) \} _{1\le i\le I}$ 
 of pairs of positive numbers. Let $Q_i^* = H_i^*V_i^*$. For \eqref{F1} to be valid as $t\to \infty$, we impose that 
 \be
 \sum_{i\in \cI _n} (-1)^{\varepsilon  _{i,n}} Q_i^* =0, \qquad \forall n\in \cN _M.
 \label{F3}
 \ee
 
 Introduce  the characteristic velocities 
 \ba
 \mu _i &=& V_i - \sqrt{gH_i}, \label{F4}\\
 \lambda _i &=& V_i + \sqrt{gH_i} \label{F5} 
 \ea
 and the Riemann invariants (see \cite{evans,HPCAB})
 \ba
 u_i &=& V_i + 2\sqrt{gH_i} -(V_i^* + 2\sqrt{gH_i^*}), \label{F6} \\
 v_i &=& V_i  - 2\sqrt{gH_i} - (V_i^* - 2\sqrt{gH_i^*}).  \label{F7}
 \ea
 We shall assume thereafter that the flow is {\em subcritical} or {\em fluvial}; that is, the characteristic
 velocities are of opposite sign
 \[
 \mu _i < 0 < \lambda _i. 
 \]
 Clearly, this holds if 
 \be
 0 < V_i^* < \sqrt{gH_i^*} \label{F8}
 \ee
 and $\max (|H_i-H_i^*|, |V_i-V_i^*|)$ is small enough. From now on, we assume that \eqref{F8} holds for 
 all $i\in \cI$, and we pick a number $c>0$ such that 
 \be
 \label{F9}
 \sqrt{gH_i^*} -V_i^* >2c, \quad \forall i\in \cI . 
 \ee
 Note that \eqref{F6}-\eqref{F7} may be inverted as 
 \ba
 H_i&=& \left( \sqrt{H_i^*} +\frac{1}{4\sqrt{g}} (u_i-v_i)\right) ^2, \label{F11}\\  
 V_i&=& V_i^* + \frac{1}{2} (u_i + v_i). \label{F12}
 \ea
Substituting the values of $H_i,V_i$ in \eqref{F6}-\eqref{F7} yields
\ba
\mu _i &=& V_i^* -\sqrt{gH_i^*} + \frac{1}{4} (u_i + 3 v_i), \\
\lambda _i &=& V_i^* + \sqrt{gH_i^*} + \frac{1}{4} (3 u_i + v_i). 
\ea
 Combined with \eqref{F9}, this shows that 
 \[
 \max (|u_i|, |v_i| ) \le c \quad \Rightarrow \quad \mu _i  < -c < c < \lambda _i.  
 \]
The shallow water equations \eqref{sw1}-\eqref{sw2}, when expressed in terms of the Riemann invariants $u_i$ and $v_i$, read
\ba
\partial _t u_i + \lambda _i(u_i,v_i) \partial _x u_i =0,&& \quad t > 0,\ 0<x<l_i, \\
\partial _t v_i + \mu _i(u_i, v_i) \partial _x v_i=0,&& \quad t >  0,\  0<x<l_i. 
\ea  
Let us now turn our attention to the boundary conditions. Consider first a boundary condition associated with an active control, e.g.
\be
\frac{\mathrm {d}}{\mathrm{dt}} v_i(t, l_i) = -K \mathrm{sgn}( v_i(t, l_i) )|v_i(t, l_i)|^\gamma.  
\label{F15}
\ee
In practice, one would like to assign the value of $Q_i(t,l_i)=H_i(t,l_i)V_i(t,l_i)$ by using the output $H_i(t,l_i)$ only. Using \eqref{F7},
it is sufficient to set
\be
\label{F16}
Q_i (t,l_i)=H_i(t,l_i)\left( v_i (t,l_i) + 2\sqrt{gH_i(t,l_i)} + V_i^*-2\sqrt{g H_i^*}  \right),  
\ee
where $v_i$ solves \eqref{F15} together with the initial condition 
\be
\label{F17}
v_i(0,l_i)= V_i(0,l_i) - 2\sqrt{ gH_i(0,l_i) } -V_i^* +2\sqrt{gH_i^*} .
\ee
For a control applied to the initial point of the $i$-edge, we set 
\be
\label{FG1}
Q_i(t,0) = H_i(t,0) \big(u_i(t,0) - 2\sqrt{gH_i(t,0)}  +  V_i^* + 2\sqrt{gH_i^*} \big), 
\ee
where $u_i(.,0)$ solves
\ba
\label{FG2} 
\frac{d}{dt} u_i(t,0) &=& -K \mathrm{sgn} (u_i(t,0)) |u_i(t,0)|^\gamma ,\\
\label{FG3}
u_i(0,0) &=& V_i(0,0) + 2\sqrt{gH_i(0,0)} -V_i^* -2\sqrt{gH_i^*}. 
\ea
Consider next a boundary condition without any active control. For a simple node $n\in \cN _{S}$ and the corresponding edge $i\in \cI _n$, if 
$\varepsilon _{i,n}=0$ (i.e. the node $n$ is the initial point of the edge $i$), then $n=i$ and a natural boundary condition at the node $n$ is given by the relation
\be
\label{G1}
Q_i(t,0)=Q_i^*,
\ee 
that is 
\be
\label{G2}
F_i(u_i(t,0),v_i(t,0))=0
\ee
where 
\[
F_i(u,v) = \big( \sqrt{H_i^*} + \frac{1}{4\sqrt{g}} (u-v)\big) ^2 \big( V_i^* + \frac{1}{2} (u+v)\big) -H_i^*V_i^*.
\]
Since 
\[
F_i(0,0)=0 \text{ and } \frac{\partial F_i}{\partial u} (0,0)= \frac{1}{2}\sqrt{H_i^*} (\sqrt{H_i^*} +\frac{V_i^*}{\sqrt{g}} ) >0
\]
it follows from the Implicit Function Theorem that there exist a number $\delta _i>0$ and a function $h_i\in C^1(\R )$ with $h_i(0)=0$ such that
for $\max (|u|,|v|)<\delta _i$, 
\[
F_i(u,v) = 0 \quad \iff \quad u=h_i(v). 
\]
Thus \eqref{G2} may be written, at least locally, in the form
\[
u_i(t,0)= h_i(v_i(t,0)). 
\]
Finally, for a multiple node $n\in \cN _{M}$, if $i_0\in \cI _n$ is the only index such that $\varepsilon _{i_0,n}=0$ (i.e. $i_0=n$), then \eqref{F2}  
may be written 
\[
F_{i_0} (U_{i_0}(t,0), v_{i_0}(t,0), U(t),V(t))=0
\]
where $U(t)=(u_i(t,l_i))_{i\in \cI _n, i\ne i_0}$, $V(t) =(v_i(t,l_i))_{i\in \cI _n, i\ne i_0}$ and 
\begin{multline*}
F_{i_0}(u_{i_0},v_{i_0}, U,V) =
\big( \sqrt{H_{i_0}^*} + \frac{1}{4\sqrt{g}} (u_{i_0}-v_{i_0})\big) ^2 \big( V_{i_0}^* + \frac{1}{2} (u_{i_0}+v_{i_0})\big) \\
-\sum_{i\in \cI _n, i\ne i_0} \big( \sqrt{H_i^*} + \frac{1}{4\sqrt{g}} (u_i-v_i)\big) ^2 \big( V_i^* + \frac{1}{2} (u_i+v_i)\big). 
\end{multline*}  
Note that, by \eqref{F3}, $F_{i_0}(0,0,0,0)=0$ and 
\[
\frac{\partial F_{i_0}}{\partial u_{i_0}}(0,0,0,0)= \frac{1}{2} \sqrt{H^*_{i_0}} (\sqrt{H^*_{i_0}} + \frac{V_{i_0}^*}{\sqrt{g}})>0. 
\]
We may pick a number $\delta _{i_0}>0$ and a function $H_{i_0}$ of class $\Cc ^1$ around $0$ such that, if $|u_i|<\delta _{i_0}$ and
$|v_i|<\delta _{i_0}$ for all $i\in \cI  _n$, we have 
\[
F_{i_0}(u_{i_0},v_{i_0},U,V) = 0 \quad \iff \quad u_{i_0} =H_{i_0} (v_{i_0}, U,V).
\]
Replacing $(u_i,v_i)$ by $(u_i(t,l_i),v(t,l_i))$ for $i\ne i_0$ in $U,V$, we see that \eqref{F2} may be written, at least locally, in the  form
\be
u_{i_0}(t,0) = h_{i_0}(v_{i_0}(t,0),t) 
\label{ZYX}
\ee
where $h_{i_0}\in \Cc ^1(\R ^2)$ and $h_{i_0}(0,t)=0$ if $u_i(t,l_i)=v_i(t,l_i)=0$ for all $i\in \cI _n\setminus \{ i_0\}$. 

We are in a position to state our results for the regulation of water flow in channels. Consider first one channel ($\cN =\{1 ,2 \}$, $\cI =\{ 1\}$)
represented by the segment $[0,l_1]$.
\begin{thm}
\label{thm11}
(Two boundary controls) Assume that \eqref{F8} holds for $i=1$, and pick any $c>0$ as in \eqref{F9}. Then there exists a number $\delta >0$
such that for all $(H_{1,0},V_{1,0})\in \Lip  ([0,l_1])^2$ with 
\be
\label{G11}
\max (||H_{1,0} -H_1^*||_{W^{1,\infty}(0,l_1)}, ||V_{1,0} -V_1^*||_{W^{1,\infty}(0,l_1)} ) <\delta ,
\ee
there exists for any $T>0$ a unique solution $(H_1,V_1)\in \Lip ([0,T]\times [0,l_1])^2$ of \eqref{sw1}-\eqref{sw3} and \eqref{F15}-\eqref{FG3}.
Furthermore, there exists a function 
$t^*(H_1^*,V_1^*,c,\delta ,K,\gamma )$ with $\lim_{\delta \to 0} t^* =  0$ such that 
\be
\label{L1}
H_1(t,x)= H_1^*,\quad V_1(t,x) = V_1^*\qquad t\ge \frac{l_1}{c} + t^*,\ x\in (0,l_1).
\ee
Finally, the equilibrium point $(H_1^*,V_1^*)$ is stable in $\Lip ([0,1])^2$ for the system \eqref{sw1}-\eqref{sw3} and \eqref{F15}-\eqref{FG3}.
\end{thm}
\begin{proof} Noticing that the map $\Theta :(H_1,V_1)\to (u_1,v_1)$  defined along \eqref{F6}-\eqref{F7} 
is locally around $(H_1^*,V_1^*)$ a diffeomorphism of class $\Cc ^\infty$, the condition
\eqref{G11} implies \eqref{eq:BornInit}-\eqref{eq:BornInitGrad} for $C_1$ and $C_2$ as in Theorem \ref{main1} (applied actually on the interval $(0,l_1)$
rather than $(0,1)$), provided that $\delta <\delta _0$ is small enough. We modify the functions $\mu_1(u,v)$ and $\lambda _1(u,v)$ outside
$[-c,c]^2$ so that 
\[
\mu  _1 (u,v) \le -c < c \le \lambda _1 (u,v),\qquad (u,v)\in \R ^2. 
\] 
Let $(u_1,v_1)$ be the solution given by Theorem \ref{main1}, and let $(H_1,V_1):=\Theta ^{-1}(u_1,v_1)$. If $C_1$ is chosen sufficiently small, then 
we infer from  \eqref{eq:direct}  that 
\begin{eqnarray*}
\max (|u_1(t,x)|, |v_1(t,x)| ) \le C_1 < c,&& \quad t\ge 0,  \ 0<x<l_1, \\
\max (|H_1(t,x) - H_1^*|, |V_1(t,x)-V_1^*|) < \delta  ,&& \quad t\ge 0, \ 0<x<l_1.
\end{eqnarray*}
It follows that for all $T>0$, $(H_1,V_1)\in \Lip ([0,T]\times [0,l_1])^2$ is a solution of \eqref{sw1}-\eqref{sw3} and \eqref{F15}-\eqref{FG3} such that 
\eqref{L1} holds with $t^*=C_1^{1-\gamma } / ((1-\gamma ) K)$. Note that the range of $C_1$ in Theorem \ref{main1} depends on  
$H_1^*$, $V_1^*$ and $c$ through the constants $M_1$ and $M_2$, and that $(C_1,C_2)\to (0,0)$ as $\delta\to 0$. Thus 
$t^*\to 0$ as $\delta \to 0$ with $K$ and $\gamma$ kept constant.  The uniqueness of 
$(H_1,V_1)$ in the class $\Lip ([0,T]\times [0,l_1])^2$ for all $T>0$ follows at once from those of $(u_1,v_1)$ in the same class, as stated in Theorem \ref{main1}. 
The stability property follows from \eqref{stab1}. 
\end{proof}

If the control is active at one endpoint of the channel only, a finite-time stabilization may be derived as well. 
\begin{thm}
\label{thm12}
(One boundary control) Assume that \eqref{F8} holds for $i=1$, and pick any $c>0$ as in \eqref{F9}. Then there exists a number $\delta >0$
such that for all $(H_{1,0},V_{1,0})\in \Lip ([0,l_1])^2$ with 
\ba
\label{G110}
&&\max (||H_{1,0} -H_1^*||_{W^{1,\infty}(0,l_1)}, ||V_{1,0} -V_1^*||_{W^{1,\infty}(0,l_1)} ) <\delta ,\\
&&H_{1,0}(l_1)V_{1,0}(l_1) = H_1^*V_1^*, \label{G110bis}
\ea
there exists for any $T>0$ a unique solution $(H_1,V_1)\in \Lip ([0,T]\times [0,l_1])^2$ of \eqref{sw1}-\eqref{sw3},  \eqref{F15}-\eqref{F17} and \eqref{G1}.
Furthermore, there exists a function 
$t^*(H_1^*,V_1^*,c,\delta ,K,\gamma )$ with $\lim_{\delta \to 0}t^*=0$  such that 
\be
\label{LL1}
H_1(t,x)= H_1^*,\quad V_1(t,x) = V_1^*\qquad t\ge \frac{2l_1}{c} + t^*,\ x\in (0,l_1).
\ee
Finally, the equilibrium point $(H_1^*,V_1^*)$ is stable in $\Lip ([0,1])^2$ for the system \eqref{sw1}-\eqref{sw3}, \eqref{F15}-\eqref{F17} and \eqref{G1}.
\end{thm}
\begin{proof}
It is sufficient to proceed as for the proof of Theorem \ref{thm11}, and to use Theorem \ref{main2} (on the domain $[0,l_1]$ and with 
a control active at the final point only).
\end{proof}

A direct application of Theorem \ref{thm12} gives the following result for a chain of two channels ($\cN =\{1 ,2 , 3\}$, $\cI =\{ 1, 2\}$), 
for which there is no active control at the internal node.
\begin{Co}
(Two channels and two controls) Assume that \eqref{F8} holds for $i=1,2$ with $Q_1^*=Q_2^*$, and pick any $c>0$ as in \eqref{F9}. Then there exists a number $\delta >0$
such that for all $(H_{1,0},V_{1,0}, H_{2,0}, V_{2,0})\in \Lip ([0,l_1])^2\times \Lip ([0,l_2])^2$ with 
\ba
\label{G120}
&&\max (||H_{i,0} -H_i^*||_{W^{1,\infty}(0,l_i)}, ||V_{i,0} -V_i^*||_{W^{1,\infty}(0,l_i)} ) <\delta , \quad i=1,2,\\
&& H_{1,0}(l_1)V_{1,0}(l_1) = Q_1^* = Q_2^* = H_{2,0}(0)V_{2,0}(0),
\ea
there exists for any $T>0$ a unique solution $(H_i,V_i)\in \Lip ([0,T]\times [0,l_1])^2$ of \eqref{sw1}-\eqref{sw3} for $i=1,2$,  \eqref{F15}-\eqref{F17} for $i=2$, 
\eqref{FG1}-\eqref{FG3} for $i=1$, and 
\[
Q_1(t,l_1)=Q_1^*=Q_2^*=Q_2(t,0), \qquad t>0. 
\]
Furthermore, there exists a function 
$t^*(H_1^*,V_1^*,H_2^*,V_2^*,c,\delta ,K,\gamma )$ with $\lim_{\delta \to 0}t^*=0$  such that 
\be
\label{LL2}
H_i(t,x)= H_i^*,\quad V_i(t,x) = V_i^*\qquad t\ge \frac{2\max(l_1,l_2)}{c} + t^*,\ x\in (0,l_i), \ i=1,2.
\ee
\end{Co}

We can extend the above results to a network of open channels which is a tree. We assume that the incoming flows can be controlled at each
multiple node  (the outgoing flow being uncontrolled and deduced from the conservation of the flows). In terms of Riemann invariants, for the edge
with index $i$, the function $v_i$ is controlled at $x=l_i$ according to \eqref{F15}, while the function $u_i$ is controlled at $x=0$ according to \eqref{FG2}
only if the initial point of the edge is a simple node (otherwise, $u_i(t,0)$  is given by \eqref{ZYX}). 

The main result of this section is the following
\begin{thm}
\label{main5}
(Network of open channels) Consider a tree with $N$ nodes and $I=N-1$ edges.
Assume that \eqref{F8} holds for $i=1,...,I$, that \eqref{F3} holds, and pick any $c>0$ as in \eqref{F9}. 
Then there exists a number $\delta >0$
such that for all $(H_{1,0},V_{1,0}, ... ,H_{I,0},V_{I,0} )\in \Lip ([0,l_1])^2\times \cdots \times \Lip ([0,l_I])^2$ with 
\ba
\label{G130}
&&\max (||H_{i,0} -H_i^*||_{W^{1,\infty}(0,l_i)}, ||V_{i,0} -V_i^*||_{W^{1,\infty}(0,l_i)} ) <\delta ,\qquad i=1,...,I,\\
&&H_{n,0} (0,0)V_{n,0}(0,0)= \sum_{i\in \cI _n, i\ne n} H_{i,0}(0,l_i)V_{i,0}(0,l_i) ,\qquad \forall n\in \cN _M , \label{G130bis}
\ea
there exists for any $T>0$ a unique function 
$(H_1,V_1, ..., H_I,V_I)\in \Lip ([0,T]\times [0,l_1])^2\times \cdots\times \Lip ([0,T]\times [0,l_I])^2$ such that, for all $i=1,...,I$, 
\eqref{sw1}-\eqref{sw3} and \eqref{F15}-\eqref{F17} hold, and \eqref{FG1}-\eqref{FG3} hold 
if the initial point of the $i$-th edge is simple, while \eqref{F2} holds
if the initial point of the $i$-th edge is multiple. 
Furthermore, there exists a function 
$t^*(H_1^*,V_1^*,...,H_I^*,V_I^*,\delta , c,K,\gamma )$ with $\lim_{\delta \to 0}t^*=0$  such that 
\be
\label{LL3}
H_i(t,x)= H_i^*,\quad V_i(t,x) = V_i^*\qquad t\ge \frac{p\max_{1\le i\le I}l_i}{c} + t^*,\ x\in (0,l_i),\  i=1,...,I,
\ee
where $p$ denotes the depth of the tree. Finally, the equilibrium state $(H_i^*,V_i^*)_{1\le i\le I}$ is stable in 
$\Lip ([0,l_1])^2\times \cdots \times \Lip ([0,l_I])^2$ for the system.
\end{thm}
\begin{proof}
The proof is done by induction on the number of edges $I\ge 1$. For $I=1$, the result was already proved in Theorem \ref{thm11}. 
Note that the norm $||(u_1,v_1)||_{L^\infty (\R ^+ ; \text{Lip} ([0,l_1])^2) } $ in Theorem \ref{thm11} is as small as desired if $\delta$ is small enough.
Let $I\ge 2$, and assume the result true for any tree with at most $I-1$ edges, with the norms
 $||(u_i,v_i)||_{L^\infty (\R ^+ ; \text{Lip} ([0,l_i])^2) } $ in the edges of the tree as small as desired if $\delta$ is small enough.  Pick any tree with $I$ edges.
 Recall that the root $\mathcal R$ is the node with index $N$, and that it is the final point of the edge of index $I=N-1$. Denote by $\mathcal R'$
the initial point of the edge of index $I$, i.e. the node of index $N-1$. 
Let $k = \# (\cI _{N-1} )$, and let us denote by $\cT _1$, ..., $\cT _{k-1}$ the subtrees of $\cT$ with $\mathcal R'$ as root. ($\mathcal R$ does not belong to 
any of them.) Note that the subsystem associated with any subtree $\cT_i$ 
is decoupled from the other subtrees and from the last edge of index $I$.
An application of the induction hypothesis on each subtree $\cT _i$, $1\le i\le k-1$, yields the existence (and uniqueness) of the functions $(H_i,V_i)$ for
$i=1,...,I-1$. Next, the existence and uniqueness of  $(H_I,V_I)$  follows at once from Theorem \ref{main2}. Indeed, the constant $D_2$ in Theorem  \ref{main2} may be taken 
as small as we want if  $\delta $ is sufficiently small, for the quantities $||\partial _t u_i(.,l_i)||_\infty$ and $||\partial _t v_i(.,l_i)||_\infty$ for $i\in {\mathcal I}_{N-1}\setminus \{ N-1\}$ 
may be taken arbitrarily small by the induction assumption. Furthermore, the norm $||(u_I,v_I)||_{L^\infty (\R ^+ ; \text{Lip} ([0,l_I])^2) } $ tends to 0 with $\delta$, by \eqref{ABC1}-\eqref{ABC2}. 
The condition \eqref{LL3} is obtained by an obvious induction on the depth of the tree. 
\end{proof}

\section*{Appendix: Proof of Proposition \ref{prop:Lip}.}
First, we introduce some extension operator $\Pi$ which maps a function $a:[0,T]\times [0,1] \to\R$ to a function
$\tilde a= \Pi (a):\R ^2\to \R$ defined as follows:
\begin{itemize}
\item for $0\le t\le T$
$$\tilde a(t,x)=\left\{ \begin{array}{ll} a(t,x) &\text{ if } x\in [0,1],\\ a(t,2-x) &\text{ if } x\in [1,2], \end{array}\right.$$
and $\tilde a(t,x)$ is 2-periodic in $x$ (i.e. $\tilde a(t,x+2)=\tilde a(t,x)$); 
\item for $t >T$, $\tilde a(t,x) = \tilde a(T,x)$ for all $x\in \R$;
\item  for $t < 0$, $\tilde a(t,x)= \tilde a(0,x)$ for all $x\in \R$;
\end{itemize}
It is easy to see that $\Pi$ is a (linear) operator from $\Cc ^0 (\Ot )$ to $\Cc ^0 (\R ^2)$ (resp. from 
$L^\infty (0,T;\text{Lip} ([0,1]))$ to $L^\infty (\R ;\text{Lip} (\R ))$ such that 
\ba
|| \Pi (a)||_{L^\infty( \R ^2 )} &=& ||a||_{\Cc  ^0(\Ot )}, \label{ext1}\\
|| \Pi (a)||_{L^\infty( \R ; \text{Lip} (\R) )} &=& ||a||_{L^\infty (0,T;\text{Lip} ( [0,1] ) )} , \label{ext2}\\
\Pi (a)(t,x)\ge c \ \  \forall (t,x)\in \R ^2 &\text{if }& a(t,x)\ge c\ \ \forall (t,x)\in\Ot . \label{ext3}
\ea
Let $a$ fulfill \eqref{A3}, and  let $\phi$ (resp. $\tilde \phi$) denote the flow associated with $a$ (resp. with $\tilde a = \pi (a)$).
Then $\tilde \phi$ is defined in $\R ^3$, and 
\be
\label{ext4}
\phi (s,t,x) =\tilde\phi (s,t,x)\qquad \forall (s,t,x)\in \text{Dom  } \phi.
\ee
Thus it is sufficient to prove that $\tilde\phi $ is $K$-Lipschitz on $[0,T]^2\times [0,1]$. To this end, pick any $(s_1,t_1,x_1),(s_2,t_2,x_2)\in [0,T]^2\times [0,1]$. 
Then 
\begin{eqnarray*}
&& |\tilde \phi (s_1,t_1,x_1)- \tilde\phi (s_2,t_2,x_2) | \\
&&\quad \le 
|\tilde \phi (s_1,t_1,x_1)- \tilde\phi (s_2,t_1,x_1) | +  |\tilde \phi (s_2,t_1,x_1)- \tilde\phi (s_2,t_2,x_1) | +  |\tilde \phi (s_2,t_2,x_1)- \tilde\phi (s_2,t_2,x_2) | \\
&&\quad =: I_1+I_2+I_3.
\end{eqnarray*}
First, 
\be
\label{WW1}
I_1 = \vert \int_{s_1}^{s_2} \partial _s \tilde \phi (\tau , t_1, x_1) d\tau \vert
= \vert \int_{s_1}^{s_2} \tilde a (\tau , \tilde \phi (\tau , t_1, x_1)) d\tau \vert  
\le || a ||_{\Cc ^0(\Ot ) } |s_1-s_2|,    
\ee
where we used \eqref{ext1}. 
For $I_2$, we notice that for all $s$ 
\begin{eqnarray*}
\vert \partial _s [\tilde\phi (s,t_1,x_1) -  \tilde \phi (s,t_2,x_1)  ] \vert &=& \vert \tilde a (s,\tilde \phi (s,t_1,x_1)) - \tilde a (s,  \tilde \phi (s,t_2,x_1)) \vert\\
&\le& L \vert \tilde\phi (s,t_1,x_1)- \tilde \phi (s,t_2,x_1) \vert 
\end{eqnarray*}
where we used \eqref{ext2}.
Gronwall's lemma combined to the estimate for $I_1$  yields then
\ba
\vert \tilde\phi (s_2,t_1,x_1) - \tilde \phi (s_2,t_2,x_1)   \vert &\le & \vert\tilde \phi (t_2,t_1,x_1)  - \tilde\phi (t_2,t_2,x_1)   \vert e^{L | s_2-t_2 | }  \nonumber \\
 &\le & \vert \tilde\phi (t_2,t_1,x_1) -  \tilde\phi (t_1,t_1,x_1)   \vert e^{L T } \nonumber \\
&\le& ||a||_{\Cc ^0 (\Ot ) } e^{LT} |t_1-t_2|. \label{WW2}
\ea
Finally, for $I_3$, we notice that for all $s$
 \begin{eqnarray*}
\vert \partial _s [\tilde\phi (s,t_2,x_1) - \tilde \phi (s,t_2,x_2)  ] \vert 
&=& \vert \tilde a (s,\tilde \phi (s,t_2,x_1)) - \tilde a (s,  \tilde \phi (s,t_2,x_2)) \vert\\
&\le& L \vert \tilde\phi (s,t_2,x_1)- \tilde \phi (s,t_2,x_2) \vert , 
\end{eqnarray*}
which, combined with Gronwall lemma, yields
\be
\vert \tilde\phi (s_2,t_2,x_1)-  \tilde \phi (s_2,t_2,x_2)   \vert \le   e^{L | s_2-t_2 | } |x_1-x_2| \le e^{LT} |x_1-x_2|.
 \label{WW3}
\ee
Then \eqref{WW0} follows at once from \eqref{WW1}-\eqref{WW3}.

\section*{Acknowledgements}
The authors were partially supported by the Agence Nationale de la Recherche, Project CISIFS,
grant ANR-09-BLAN-0213-02.

\end{document}